\documentclass[10pt,a4paper,final]{article}

\pdfoutput=1
\usepackage[latin1]{inputenc}
\usepackage[english,francais]{babel}
\usepackage{amsmath,amsthm,amsfonts,amssymb}
\usepackage{a4wide,ifpdf,aeguill}
\usepackage{float,color,bm,graphicx,calc,hyperref}

\title{Existence d'un feuilletage positivement transverse à un homéomorphisme de surface}
\author{Olivier \textsc{Jaulent}\footnote{Lycée François I$^{\text{er}}$, 11 rue Victor Hugo 77\,300 Fontainebleau, \texttt{Prenom.Nom@normalesup.org}}}
\hypersetup{pdfauthor={Olivier Jaulent},pdftitle={Existence d'un feuilletage positivement transverse à un homéomorphisme de surface}}

\theoremstyle{plain}
\newtheorem{Prop}{Proposition}[section]
\newtheorem{Th}[Prop]{Théorème}
\newtheorem{Cor}[Prop]{Corollaire}
\newtheorem{Lemme}[Prop]{Lemme}

\newtheorem*{Prop*}{Proposition}
\newtheorem*{Th*}{Théorème}
\newtheorem*{Cor*}{Corollaire}
\newtheorem*{Lemme*}{Lemme}
\newtheorem*{Dep*}{Proposition-Définition}
\newtheorem*{Conj*}{Conjecture}

\theoremstyle{definition}
\newtheorem{Def}[Prop]{Définition}

\newtheorem{PropDef}[Prop]{Proposition et définition}
\newtheorem{Rem}[Prop]{Remarque}

\newtheorem{Qu}[Prop]{Question}

\newtheorem*{Def*}{Définition}
\newtheorem*{DefProp*}{Définition et proposition}
\newtheorem*{PropDef*}{Proposition et définition}
\newtheorem*{Rem*}{Remarque}
\newtheorem*{Ex*}{Exemple}
\newtheorem*{Qu*}{Question}
\newtheorem*{Not*}{Notation}

\newcommand{\id}{\operatorname{id}}
\newcommand{\N}{\mathbb{N}}
\newcommand{\R}{\mathbb{R}}
\newcommand{\Z}{\mathbb{Z}}
\newcommand{\C}{\mathbb{C}}
\newcommand{\mc}{\mathcal}
\newcommand{\Fix}{\operatorname{Fix}}
\newcommand{\Homeo}{\operatorname{Homeo}}
\newcommand{\T}{\mathbb T}
\newcommand{\A}{\mathbb A}
\newcommand{\Cont}{\operatorname{Cont}}
\newcommand{\pprec}{\preccurlyeq}
\newcommand{\Pro}[1]{\textrm{(P#1)}}
\newcommand\ent[2]{\mbox{$\{#1,\ldots,#2\}$}}

\newcommand{\qt}[1]{\quad\text{#1}\quad}
\newcommand{\qqt}[1]{\qquad\text{#1}\qquad}
\newcommand{\fonction}[4]{\begin{tabular}{c@{~}c@{~}c}
  $#1$&$\to$&$#2$\\
  $#3$&$\mapsto$&$#4$\\
\end{tabular}}

\AtBeginDocument{
 \newcommand{\FigInclusion}[1]{
      \ifpdf\includegraphics{#1.pdf}
        \else\includegraphics{#1.eps}
      \fi
 }
 \newcommand{\Fig}[3]{
  \begin{figure}
    \begin{center}
      \FigInclusion{#1}
    \end{center}
    \caption{#3}\label{#2}
  \end{figure}
 }
}

\newcommand\refbis[1]{\ref{#1}'}
\newcommand\PropBis[1]{
  \theoremstyle{plain}
  \newtheorem*{Prop#1}{Proposition~\ref{Prop:#1}'}
}

\begin{document}

\maketitle
\selectlanguage{english}

\begin{abstract}
 \textbf{Existence of a topological foliation transverse to the dynamics of a homeomorphism.}
 Let $F$ be a homeomorphism of an oriented surface $M$ that is isotopic to the identity.
 Le Calvez proved that if $F$ admits a lift $\tilde F$ without fixed points to the universal covering of $M$, then there exists a topological foliation of $M$ transverse to the dynamics.
 We generalize this result to the case where $\tilde F$ has fixed points.
 We obtain a singular topological foliation whose singularities are fixed points of $F$.


\medskip\noindent\textbf{Keywords :}
Surface homeomorphism, topological foliation transverse to the dynamics, equivariant version of Brouwer's plane translation theorem
\end{abstract}

\selectlanguage{francais}

\begin{abstract}
 Le Calvez a montré que si $F$ est un homéomorphisme isotope à l'identité d'une surface $M$ admettant un relèvement $\tilde F$ au revêtement universel n'ayant pas de points fixes, alors il existe un feuilletage topologique de $M$ transverse à la dynamique.
 Nous montrons que ce résultat se généralise au cas où $\tilde F$ admet des points fixes.
 Nous obtenons alors un feuilletage topologique singulier transverse à la dynamique dont les singularités sont  un ensemble fermé de points fixes de $F$.

\medskip\noindent\textbf{Mots-clés :}
Homéomorphisme de surface, feuilletage topologique transverse à la dynamique, version équivariante feuilletée du théorème de translation plane de Brouwer
\end{abstract}

\bigskip

Dans cet article, nous considérons une surface connexe $M$ et le groupe $\Homeo_*(M)$ des homéomorphismes de $M$ isotopes à l'identité.
Nous appelons surface une variété topologique de dimension deux sans bord, orientable, séparée et dénombrable à l'infini.
Le résultat principal est le suivant :

\begin{Th} \label{Th:princ}
 Soient $F \in \Homeo_*(M)$ et $I = (F_t)_{t \in [0,1]}$ une isotopie entre l'identité et $F$.
 Alors, il existe un ensemble fermé $X$ de points fixes de $F$ et
 un relèvement $\widehat F$ de la restriction $F_{|M \setminus X}$ au revêtement universel $\widehat\pi_X : \widehat N_X \rightarrow M \setminus X$ de $M \setminus X$ vérifiant :
 \begin{enumerate}
  \item
  $\widehat F$ commute avec les automorphismes du revêtement universel et est sans point fixe ;
  \item
  tout chemin de $\widehat N_X$ joignant un point $\widehat x \in \widehat N_X$ à son image $\widehat F(\widehat x)$ se projette sur $M$ en un chemin homotope à extrémités fixées au chemin $I(x)$ décrit par le projeté $x = \widehat\pi_X(\widehat x)$ le long de l'isotopie $I$ ;
  
  \item
  pour tout sous-ensemble fermé $Y \subset X$, il existe un relèvement $\widehat F_Y$ de la restriction $F_{|M \setminus Y}$ au revêtement universel $\widehat\pi_Y : \widehat N_Y \rightarrow M \setminus Y$ de $M \setminus Y$ vérifiant :
 \begin{enumerate}
  \item
  $\widehat F_Y$ commute avec les automorphismes du revêtement universel de $M \setminus Y$ et fixe les antécédents par $\widehat \pi_Y$ des points de $X \setminus Y$ ;
  \item
  tout chemin de $\widehat N_X$ joignant un point $\widehat x \in \widehat N_X$ à son image  $\widehat F(\widehat x)$ se projette sur $M$ en un chemin dont tout relèvement à $\widehat N_Y$ issu d'un point $\widehat y$ a pour extrémité $\widehat F_Y(\widehat y)$.
 \end{enumerate}
 \end{enumerate}
\end{Th}

Justifions la présence des différentes conditions qui apparaissent dans cet énoncé.
On notera $\Fix(I)$ l'ensemble des points fixes le long d'une isotopie $I$.
Une idée pour obtenir le relèvement $\widehat F$ est de construire une isotopie $I = (F_t)_{t \in [0,1]}$ entre l'identité et $F$ et de considérer $X = \Fix(I)$.
En effet, l'isotopie restreinte $\bigl((F_t)_{|M\setminus X}\bigr)_{t \in [0,1]}$ se relève au revêtement universel de $M \setminus X$ en une isotopie $(\widehat F_t)_{t \in [0,1]}$ issue de l'identité de temps un le relèvement $\widehat F = \widehat F_1$ de $F_{|M \setminus X}$.
Il est facile de vérifier que $\widehat F$ commute avec les automorphismes de revêtement et vérifie les assertions 2 et 3 du théorème~\ref{Th:princ}.

Il reste à vérifier la principale propriété : on souhaite que $\widehat F$ soit sans point fixe.
La proposition \ref{PropTh:XyNonEnlace} assure que c'est le cas si l'isotopie $I$ est maximale pour la relation d'ordre suivante.
Si $I$ et $J$ sont deux isotopies entre l'identité et $F$, on dit que l'on a $I \leq J$ si $\Fix(I)$ est contenu dans $\Fix(J)$ et si $I$ et $J$ sont homotopes relativement à $\Fix(I)$.
Malheureusement nous ne savons pas répondre à la question suivante.

\begin{Qu}
Existe-t-il une isotopie maximale pour la relation d'ordre $\leq$ ?
\end{Qu}

Nous avons donc laissé de côté les isotopies définies sur $M$ au profit de la notion plus faible suivante.
On dit qu'un sous-ensemble fermé $X \subset M$ est \textit{non enlacé} s'il existe un relèvement $\widehat F$ de la restriction $F_{|M\setminus X}$ au revêtement universel $\widehat\pi_X : \widehat N_X \to M \setminus X$ de $M \setminus X$ qui commute avec les automorphismes de revêtement.
Nous donnerons dans le paragraphe \ref{ssec:defenlac} plus de détails sur cette définition.
Nous introduirons ensuite une relation d'ordre $\pprec$ (définition \ref{Def:RO3}) sur les ensembles non enlacés qui permettra d'obtenir la seconde assertion de notre énoncé.
Nous montrerons qu'il existe des ensembles non enlacés maximaux pour cette relation d'ordre (proposition \refbis{Prop:existmax}) et que les relèvements associés sont sans point fixe (proposition \refbis{Prop:MaximaFaible}) ce qui complète la première assertion.

Justifions maintenant la présence de la troisième.
Dans certaines situations où $X$ sépare $M$, il peut apparaître des composantes connexes de $M \setminus X$ homéomorphes à l'anneau $(\R/\Z)\times \R$.
Lorsque cela se produit, il n'y a pas unicité d'un relèvement $\widehat F$ qui commute avec les automorphismes de revêtement.
Or en général l'un de ces relèvements présente des propriétés supplémentaires et la troisième assertion du théorème permet de le caractériser.
Notre preuve consiste à rechercher un ensemble non enlacé $X$ maximal pour $\pprec$ et il semble raisonnable d'imposer qu'il soit plus grand (pour $\pprec$) que ses sous-ensembles fermés, ce qui correspond à la troisième assertion.
Il se trouve que cette condition suffit pour assurer l'unicité du relèvement.
En effet, certaines parties fermées $Y \subset X$ ne sépareront pas~$M$.
On obtiendra alors l'unicité d'un relèvement $\widehat F_Y$ de $F_{|M\setminus Y}$ puis l'unicité d'un relèvement $\widehat F$ de $F_{|M \setminus X}$ plus grand que $\widehat F_Y$ pour $\pprec$.
Nous donnons un exemple explicite dans la partie \ref{sec:pres}.

\begin{Rem}
 Un théorème d'Epstein (voir paragraphe \ref{ssec:defenlac} et \cite{Epstein66}) permet de reformuler le théorème \ref{Th:princ} en terme d'isotopie sur $M \setminus X$.
 Il affirme en effet que l'existence de $\widehat F$ entraîne celle d'une isotopie sur $M \setminus X$ entre l'identité de $M \setminus X$ et la restriction $F_{|M\setminus X}$ dont le relèvement au revêtement universel $\widehat N_X$ de $M \setminus X$ issu de l'identité admet $\widehat F$ pour temps un.
\end{Rem}

L'un des principaux intérêts du théorème \ref{Th:princ} est que l'on peut en déduire l'existence d'un feuilletage positivement transverse à un homéomorphisme de surface (corollaire \ref{Cor:princ}).
Nous allons maintenant expliquer comment.
La preuve repose sur la version équivariante feuilletée suivante du théorème de Brouwer, un résultat difficile démontré par Le Calvez.

\begin{Th}[Le Calvez \cite{LC05}] \label{Th:Brouwer}
 Soit $\widehat G$ un groupe discret d'homéomorphismes préservant l'orientation, agissant librement et proprement sur le plan $\R^2$.
 Soit $\widehat F$ un homéomorphisme de $\R^2$ préservant l'orientation et sans point fixe qui commute avec tous les éléments de $\widehat G$.
 Alors il existe un feuilletage topologique $\widehat{\mc F}$ de $\R^2$ par des droites de Brouwer
de~$\widehat F$ tel que $\widehat{\mc F}$ soit invariant par tout élément de $\widehat G$.
\end{Th}

Une \textit{droite topologique orientée} est un plongement topologique propre $\phi$ de la droite réelle orientée $\R$ ; elle sépare $\R^2$ en deux composantes connexes.
On dit qu'un tel plongement $\phi$ est une \textit{droite de Brouwer} de $\widehat F$ si $\widehat F(\phi)$ est contenu dans la composante connexe de droite et $\widehat F^{-1}(\phi)$ dans celle de gauche.

On déduit de ce résultat que sous les hypothèses du théorème \ref{Th:princ} il existe un feuilletage topologique $\widehat{\mc F}$ de $\widehat N_X$ par des droites de Brouwer de $\widehat F$, tel que $\widehat{\mc F}$ soit invariant par les automorphismes du revêtement $\widehat\pi_X : \widehat N_X \rightarrow N_X$.
Le feuilletage $\widehat{\mc F}$ passe donc au quotient en un feuilletage topologique $\mc F$ de~$M$.
Ce feuilletage est dit \textit{positivement transverse} à $\widehat F$.
Expliquons ce que cela signifie avant d'énoncer le corollaire.

\begin{Def}[Chemin transverse à un feuilletage]
 Soit $M$ une surface et $\mc F$ un feuilletage orienté sur $M$.
 On dit qu'un chemin $\gamma : [0,1] \rightarrow\nobreak M$ est \textit{négativement transverse au feuilletage $\mc F$} si pour tout $t_0 \in ]0,1[$, il existe un intervalle ouvert $I_0 \subset ]0,1[$ contenant $t_0$, un voisinage $V$ de $\gamma(t_0)$ dans~$M$ et un homéomorphisme $\varphi : V \rightarrow ]-1,1[^2$ qui envoie $\gamma(t_0)$ sur le point $(0,0)$, l'ensemble $V \cap \gamma\bigl(I_0\bigr)$ sur l'arc horizontal $]-1,1[ \times \{0\}$ orienté vers la droite et le feuilletage $\mc F$ sur les droites verticales orientées vers le haut.
\end{Def}

\begin{Rem}
 Si $X \subset \Fix(F)$ est un ensemble fermé de points fixes, un feuilletage sur $M \setminus X$ peut être vu comme un feuilletage topologique singulier sur $M$ dont les singularités sont les points de $X$.
 On dira qu'un chemin $\gamma : [0,1] \to M \setminus X$ est transverse au feuilletage singulier sur $M$ s'il est transverse au feuilletage sur $M \setminus X$.
\end{Rem}

\begin{Def}[Feuilletage transverse à la dynamique] \label{Def:postrans}
 Soit $F$ un homéomorphisme d'une surface $M$, admettant un relèvement $\widetilde F$ au revêtement universel $\widetilde M$ de $M$ qui commute avec les automorphismes de revêtement.
 Soit $\mc F$ un feuilletage singulier sur $M$ dont les singularités sont des  points fixes de $F$.
 Le feuilletage $\mc F$ est dit \textit{positivement transverse à $\widetilde F$} si pour tout point $z \in M \setminus X$, il existe un chemin $\gamma_z$ de $z$ à $F(z)$ dans $M \setminus X$, négativement transverse au feuilletage $\mc F$ et \textit{associé à $\widetilde F$},
  c'est-à-dire se relevant à $\widetilde M$ en un chemin joignant son origine $\widetilde z$ à $\widetilde F(\widetilde z)$ (voir définition~\ref{Def:Associe}).
\end{Def}

\begin{Cor} \label{Cor:princ}
 Soient $F \in \Homeo_*(M)$ et $\widetilde F$ un relèvement à $\widetilde M$ qui commute avec les automorphismes de revêtement.
 Alors, il existe
  un sous-ensemble fermé $X \subset \Fix(F)$
  et un feuilletage topologique $\mc F$ de $M$, dont $X$ est l'ensemble des singularités, qui soit positivement transverse à~$\widetilde F$.
\end{Cor}

\begin{proof}
 On considère le sous-ensemble fermé $X \subset \Fix(F)$ et le relèvement $\widehat F$ de $F_{|M\setminus X}$ donnés par le théorème \ref{Th:princ}.
 On vient de voir que le théorème \ref{Th:Brouwer} assure alors l'existence d'un feuilletage topologique $\widehat{\mc F}$ de $\widehat N_X$ par des droites de Brouwer de $\widehat F$ qui passe au quotient en un feuilletage topologique singulier $\mc F$ de $M$, de singularités $X$.

Il reste à prouver que $\mc F$ est transverse à~$\widetilde F$ ce qui, d'après la deuxième assertion du théorème~\ref{Th:princ}, revient à prouver que $\widehat{\mc F}$ est transverse à~$\widehat F$.
Résumons la preuve donnée dans \cite{LC05} page 4 : considérons un point $\widehat z_0 \in \widehat N_X$ et notons $\widehat W$ l'ensemble des points $z \in \widehat N_X$ extrémité d'un arc issu de $\widehat z_0$ et positivement transverse à $\widehat F$.
On remarque alors que la partie $\widehat W$, ouverte et \og située à droite \fg{} des feuilles qui constituent sa frontière, est le demi-plan topologique ouvert \og situé à droite \fg{} de la droite de Brouwer passant par $\widehat z_0$.
Par définition d'une droite de Brouwer, $\widehat W$ contient donc $\widehat F(\widehat z_0)$.
\end{proof}

Expliquons l'intérêt de ce résultat.
Dans \cite{LCDuke06}, Le Calvez démontre l'existence de points périodiques de périodes arbitrairement grandes pour tout homéomorphisme hamiltonien non trivial d'une surface compacte orientée de genre $g \geq 1$.
Dans sa preuve, il est amené à envisager deux cas selon que l'ensemble des points fixes peut jouer le rôle de $X$ dans le théorème \ref{Th:princ} ou non.
Or les arguments utilisés dans le premier cas s'appliquent mot pour mot en remplaçant l'ensemble des points fixes par un sous-ensemble $X$ donné par le théorème \ref{Th:princ}.
L'étude du second cas n'est donc plus nécessaire.

On peut espérer que le théorème \ref{Th:princ} et son corollaire \ref{Cor:princ} seront d'une grande utilité dans l'étude des homéomorphismes de surface.

\bigskip

Dans une première partie, nous allons rappeler des définitions et résultats classiques puis introduire des notions d'enlacement et de chemins adaptés, moins classiques et plus techniques, qui faciliteront les démonstrations des parties suivantes.
Nous démontrerons ensuite le théorème \ref{Th:princ} dans la partie~\ref{sec:pres} en nous appuyant sur les deux résultats suivants :
\begin{enumerate}
 \item
 l'existence d'un relèvement maximal dans $(\mc R,\pprec)$, ensemble des relèvements qui commutent avec les automorphismes de revêtement (voir les définitions \ref{Def:QuiCommute} et \ref{Def:RO3}).
 Il s'agit de la proposition \ref{Prop:existmax}, un des principaux résultats de l'article, qui sera prouvée dans la partie~\ref{sec:induc}.
 La démonstration de ce résultat qui repose sur le lemme de Zorn concentre l'essentiel des difficultés de l'article.
 \item
 la possibilité de modifier une isotopie admettant un point fixe contractile en une isotopie fixant ce point (proposition \ref{PropTh:XyNonEnlace}).
 Ce résultat plus simple et classique sera démontré dans la partie \ref{sec:contrac}.
 Nous avons choisi d'en donner une preuve détaillée dans la mesure où nous n'en connaissons pas de démonstration complète dans la littérature.
\end{enumerate}

Une partie de ce travail a été faite au Laboratoire d'Analyse Géométrie et Applications, à l'université Paris 13.
Je remercie très chaleureusement Patrice Le Calvez ; les nombreuses discussions que nous avons eues m'ont guidé pour introduire les outils utilisés et ses précieuses remarques ont largement contribué à clarifier les preuves.
Je remercie particulièrement François Béguin qui m'a suggéré d'ajouter la troisième condition du théorème principal en vue des applications, le rapporteur dont les nombreux commentaires ont permis de préciser le texte ainsi que Marc Bonino, Sylvain Crovisier et Frédéric Le Roux.

\section{Généralités}

\subsection{Notations}

\paragraph{Chemin, lacet.}
Soit $E$ un espace topologique.
On désigne par \textit{chemin} toute classe d'application continue $\gamma : [0,1] \rightarrow E$ modulo un reparamétrage préservant l'orientation ; le point $\gamma(0)$ est l'origine et le point $\gamma(1)$ l'extrémité.
Nous noterons $\gamma^-$ le chemin obtenu en changeant l'orientation de $\gamma$.
La \textit{concaténation} $\alpha.\beta$ de deux chemins $\alpha$ et $\beta$ vérifiant $\alpha(1) = \beta(0)$ est le chemin $\gamma$ représenté par le paramétrage suivant :
\[
 \gamma(t) =
 \begin{cases}
  \alpha(2t) &\text{si $0 \leq t \leq \frac12$ ;} \\
  \beta(2t-1) &\text{si $\frac12 \leq t \leq 1$.}
 \end{cases}
\]

Un \textit{lacet} est un chemin dont l'extrémité $z$ est égale à l'origine ; on dit alors qu'il est \textit{basé en $z$}.
En notant $\T = \R/\Z$, on peut représenter un lacet par une application continue $\gamma : \T \rightarrow E$.
On dit que deux lacets $\alpha$ et $\beta$ sont \textit{librement homotopes} s'il existe une application continue $\phi : [0,1] \times \T \rightarrow E$ vérifiant pour tout $t \in \T$, $\phi(0,t) = \alpha(t)$ et $\phi(1,t) = \beta(t)$.
L'application $\phi$ est appelée \textit{homotopie libre} entre $\alpha$ et $\beta$.
Pour tout $t_0 \in \T$, en posant $z = \phi(0,t_0)$, on note $\phi(z) : s \mapsto \phi(s,t_0)$ la \textit{trajectoire} de $z$ le long de l'homotopie libre.

Deux chemins $\alpha$ et $\beta$ vérifiant $\alpha(0) = \beta(0)$ et $\alpha(1) = \beta(1)$ sont \textit{homotopes}\footnote{Cette notion d'homotopie est souvent désignée par \textit{homotopie relativement aux extrémités}.} s'il existe une application continue $\phi : [0,1]^2 \rightarrow E$ vérifiant pour tout $(s,t) \in [0,1]^2$ :
\begin{align*}
 \phi(0,t) &= \alpha(t),
 &\phi(1,t) &= \beta(t), \\
 \phi(s,0) &= \alpha(0) = \beta(0)
 &\text{et}\qquad\phi(s,1) &= \alpha(1) = \beta(1).
\end{align*}
On parlera de même d'\textit{homotopie} entre lacets de même point base lorsque ce point reste fixe le long de l'homotopie.
Un lacet basé en un point $z \in E$ est dit \textit{contractile} s'il est homotope au lacet trivial basé en $z$.

\paragraph{Relèvement d'une isotopie, homotopie entre isotopies.}
Soit $M$ une surface connexe.
Une \textit{isotopie} $I$ est un chemin $t \mapsto F_t$ de $[0,1]$ dans $\Homeo(M)$ muni de la topologie compacte ouverte pour les homéomorphismes.
Pour tout point $z \in M$, on note $I(z) : t \mapsto F_t(z)$ la \textit{trajectoire} de $z$ le long de l'isotopie.
Il s'agit d'un lacet si et seulement si $z$ est un point fixe de $F_1$.

Soit $\widetilde\pi : \widetilde M \rightarrow M$ le revêtement universel de $M$.
Une isotopie $I=(F_t)_{t \in [0,1]}$ telle que $F_0$ soit l'homéomorphisme identité $\id_M$ se relève à $\widetilde M$ en une isotopie $\widetilde I = (\widetilde F_t)_{t \in [0,1]}$ vérifiant $\widetilde F_0 = \id_{\widetilde M}$.
L'homéomorphisme $\widetilde F_1$ est appelé relèvement de $F_1$ \textit{associé} à $I$ ; nous le noterons $\widetilde F_I$.

Considérons maintenant un sous-ensemble $N \subset M$.
On dit que $I$ et $I'$ sont \textit{homotopes relativement à $N$} si l'on peut choisir l'homotopie $\phi : [0,1]^2 \rightarrow \Homeo(M)$ de sorte que la restriction $\phi(s,t)_{|N}$ soit indépendante de $s \in [0,1]$.

Nous utiliserons également la notion de \textit{composée de deux isotopies} : si $I = (F_t)_{t \in [0,1]}$ et $J = (G_t)_{t \in [0,1]}$ sont deux isotopies, on note $J \circ I$ l'isotopie définie par $J \circ I = (G_t\circ F_t)_{t \in [0,1]}$.

\paragraph{Point fixe contractile, point fixe d'une isotopie.}

Dans la suite de l'article, on se donne une surface connexe $M$ et un homéomorphisme $F$ de~$M$ isotope à l'identité.
Si $E$ est un ensemble, on note $\overline E$ son adhérence, $\partial E$ sa frontière et $\mathring E$ son intérieur.
La lettre $X$ désignera toujours un sous-ensemble fermé de l'ensemble $\Fix(F)$ des points fixes de $F$.
Nous noterons $N_X = M \setminus X$ son complémentaire et $\widehat\pi_X : \widehat N_X \rightarrow N_X$ son revêtement universel.
Soient $I = (F_t)_{t \in [0,1]}$ une isotopie entre l'identité et $F$ et $z$ un point fixe de~$F$.

\begin{Def}
 On dit que $z$ est un \textit{point fixe de $I$} s'il  vérifie, pour tout $t \in [0,1]$, $F_t(z) = z$.
 Nous noterons $\Fix(I)$ l'ensemble des points fixes d'une isotopie $I$.
\end{Def}

\begin{Def}
 On dit que $z$ est un \textit{point contractile de $I$} s'il vérifie l'une des deux propriétés suivantes :
 \begin{enumerate}
  \item le point $z$ est un point fixe de $I$ ;
  \item le point $z$ appartient à $\Fix(F) \setminus \Fix(I)$ et le lacet $I(z)$ est contractile dans $M \setminus \Fix(I)$.
 \end{enumerate}
 Nous noterons $\Cont(I)$ l'ensemble des points contractiles d'une isotopie $I$.
\end{Def}

\begin{Rem}
 Considérons $X = \Fix(I)$ et $\widehat F$ le relèvement de $F_{|N_X}$ à $\widehat N_X$ associé à la restriction de $I$ à $N_X$.
 Un point $z \in N_X$ est contractile si et seulement si tout relevé $\widehat z$ de $z$ à $\widehat N_X$ est un point fixe de $\widehat F$.
\end{Rem}

\subsection{Isotopie restreinte, relèvement}
\label{ssec:defrestr}

\begin{Def}
 On appelle \textit{isotopie restreinte} tout couple $(X,I)$ formé :
 \begin{itemize}
  \item d'un sous-ensemble fermé $X \subset \Fix(F)$ ;
  \item d'une isotopie définie sur $N_X$ entre l'identité $\id_{N_X}$ et la restriction $F_{|N_X}$ de $F$ à $N_X$.
 \end{itemize}
 On note $\mc I$ l'ensemble des isotopies restreintes.
\end{Def}

\begin{Def} \label{Def:QuiCommute}
 On note $\mc R$ l'ensemble des couples $(X,\widehat F)$ formés :
 \begin{itemize}
  \item d'un sous-ensemble fermé $X \subset \Fix(F)$ ;
  \item d'un relèvement $\widehat F$ de $F_{|N_X}$ au revêtement universel $\widehat\pi_X : \widehat N_X \rightarrow N_X$ qui commute avec les automorphismes de ce revêtement.
 \end{itemize}
\end{Def}

\begin{Rem} \label{Rem:UniCommu}
 Soit $N$ une surface connexe non homéomorphe à l'anneau $\A = \R/\Z\times\R$ ou au tore $\T = \R/\Z \times \R/\Z$.
 Alors le groupe des automorphismes du revêtement universel de $N$ est de centre trivial (\cite{Epstein66}, Lemma 4.3).
 Si $F$ est un homéomorphisme de $N$, il admet donc au plus un relèvement $\widehat F$ au revêtement universel qui commute avec les automorphismes de revêtement.
\end{Rem}

 Soit $(X,I) \in \mc I$ une isotopie restreinte.
 Relevons $I$ à $\widehat N_X$ en une isotopie~$\widehat I$ issue de l'identité $\id_{\widehat N_X}$.
 Le temps un de~$\widehat I$ est le relèvement $\widehat F_{I}$ associé à $I$ de la restriction $F_{|N_X}$.
 Or l'identité $\id_{\widehat N_X}$ commute avec les éléments du groupe $\widehat G$ des automorphismes du revêtement $\widehat\pi_X : \widehat N_X \rightarrow N_X$.
 Puisque $\widehat G$ est discret, la propriété de commutation est conservée le long de l'isotopie~$\widehat I$ et $\widehat F_I$ commute avec les éléments de $\widehat G$.

\begin{Def}[Relèvement associé à une isotopie restreinte] \label{Def:RelAsso}
 Soit $(X,I) \in \mc I$ une isotopie restreinte.
 Le relèvement $(X,\widehat F_{I}) \in \mc R$ est appelé \textit{relèvement associé à l'isotopie restreinte $(X,I)$}.
\end{Def}

Disposer d'un relèvement $\widehat F$ de $F$ qui commute avec les automorphismes de revêtement assure que les notions suivantes ne dépendent pas du relèvement $\widehat\alpha$ ou $\widehat\phi$ choisi.

\begin{Def}[Chemin et homotopie associés à un relèvement] \label{Def:Associe}
 Soit $(X,\widehat F) \in \mc R$.
 \begin{enumerate}
  \item Un chemin $\alpha \subset N_X$ est dit \textit{associé à $\widehat F$} s'il se relève à $\widehat N_X$ en un chemin $\widehat\alpha$ dont l'extrémité est image par $\widehat F$ de son origine.
  \item Soit $\phi : [0,1] \times \T \rightarrow N_X$ une homotopie libre entre un lacet $\Gamma \subset N_X$ et son image $F(\Gamma)$.
  Choisissons un relevé $\widehat\Gamma$ de $\Gamma$ à $\widehat N_X$.
  L'homotopie libre $\phi$ est dite \textit{associée à $\widehat F$} si elle se relève en $\widehat\phi : [0,1]^2 \to \widehat N_X$ vérifiant :
  \[
   \forall t \in [0,1] \qquad
   \widehat\phi(0,t) = \widehat\Gamma(t)
   \qqt{et}
   \widehat\phi(1,t) = \widehat F \circ \widehat\Gamma(t).
  \]
 \end{enumerate}
\end{Def}

\begin{Rem}\label{Rem:Associe}
 Soit $\phi : [0,1] \times \T \to N_X$ une homotopie libre entre un lacet $\Gamma \subset N_X$ et son image $F(\Gamma)$.
 S'il existe $z_0  = \phi(0,t_0) \in \Gamma$ tel que $\phi(z_0) : s \mapsto \phi(s,t_0)$ soit un chemin associé à $\widehat F$, autrement dit $\widehat\phi(1,t_0) = \widehat F(\widehat\Gamma(t_0)) = \widehat F(\widehat\phi(0,t_0))$, alors $\phi$ est associée à $\widehat F$.
\end{Rem}

\subsection{Chemin adapté à un ensemble de points fixes}
\label{ssec:Adapt}

Nous allons maintenant caractériser les chemins associés à un relèvement $(X, \widehat F) \in \mc R$.
Nous parlerons de chemins adaptés à l'ensemble $X$.

\begin{Def}[Chemin adapté] \label{Def:Adapte}
 Soient $F \in \Homeo_*(M)$ et $X \subset \Fix(F)$ un sous-ensemble fermé.
 Un chemin $\alpha$ d'origine un point $z$ de $N_X$ et d'extrémité $F(z)$ est dit \textit{adapté à $X$} si pour tout lacet $\gamma \subset N_X$ basé en $z$, le lacet $\alpha.F(\gamma).\alpha^{-}$ est homotope, dans $N_X$, au lacet~$\gamma$.
\end{Def}

\begin{Rem} \label{Rem:AdapteHomo}
 Si $\alpha \subset N_X$ est un chemin adapté à $X$ d'origine $z \in N_X$, alors tout lacet $\gamma$ basé en $z$ est librement homotope à son image $F(\gamma)$.
 En effet, les lacets $\gamma$ et $\alpha.F(\gamma).\alpha^{-}$ sont alors homotopes et il est clair que ce dernier lacet est librement homotope à $F(\gamma)$.
\end{Rem}

\begin{Prop} \label{Prop:PlusPrecis}
 Considérons $F \in \Homeo_*(M)$ et un ensemble fermé $X \subset \Fix(F)$.
 \begin{enumerate}
  \item Supposons que l'on dispose d'un relèvement $(X,\widehat F) \in \mc R$.
  Tout chemin de $N_X$ associé à $\widehat F$ est adapté à $X$.
  \item Inversement, supposons que l'on dispose d'un chemin $\alpha$ adapté à $X$ et qu'on le relève en un chemin $\widehat\alpha \subset \widehat N_X$.
  Si $N_X$ est connexe\footnote{Dans le cas où $N_X$ n'est pas connexe, il faut disposer d'un chemin adapté par composante connexe pour construire le relèvement $\widehat F$.
Si c'est le cas, la construction est possible et le relèvement obtenu commute avec les automorphismes du revêtement universel $\widehat\pi_X : \widehat N_X \rightarrow N_X$.}%
, le relèvement $\widehat F$ de $F_{|N_X}$ qui envoie l'origine de $\widehat\alpha$ sur son extrémité commute avec les automorphismes du revêtement universel $\widehat\pi_X : \widehat N_X \rightarrow N_X$.
  Le chemin $\alpha$ est alors associé à $\widehat F$ et $(X, \widehat F)$ appartient à $\mc R$.
 \end{enumerate}
\end{Prop}

\begin{proof}
 \begin{enumerate}
  \item
  On suppose qu'il existe $(X, \widehat F) \in \mc R$.
  Soit $\alpha$ un chemin de $N_X$ associé à $\widehat F$ ;
  choisissons un relèvement $\widehat \alpha \in \widehat N_X$ de $\alpha$ et notons $\widehat z$ son origine.
  Son extrémité est donc $\widehat F(\widehat z)$.
  Considérons un lacet $\gamma \subset N_X$ basé en $z$ et relevons-le en un chemin $\widehat \gamma \subset \widehat N_X$ d'origine $\widehat z$.
  Il existe un unique automorphisme de revêtement $\widehat T$ tel que l'extrémité de $\widehat\gamma$ soit $\widehat T(\widehat z)$.
  Le chemin $\widehat F(\widehat\gamma)$ a pour origine $\widehat F(\widehat z)$ et pour extrémité $\widehat F \circ \widehat T(\widehat z)$ qui n'est autre que $\widehat T \circ \widehat F(\widehat z)$, extrémité du chemin $\widehat T(\widehat\alpha)$.
  On peut donc concaténer les chemins $\widehat\alpha$, $\widehat F(\widehat\gamma)$ et $\bigl( \widehat T(\widehat\alpha)\bigr)^{-}$.
  On obtient un chemin de mêmes extrémités $\widehat z$ et $\widehat T(\widehat z)$ que $\widehat\gamma$ ; il est donc homotope à $\widehat\gamma$ car $\widehat N_X$ est simplement connexe.
  On en déduit que les lacets $\alpha.F(\gamma).\alpha^{-}$ et $\gamma$ sont homotopes (à point base fixé) dans $N_X$ et donc que $\alpha$ est adapté à $X$.
  
  \item
  Inversement, soit $\widehat\alpha \subset \widehat N_X$ le relèvement d'un chemin $\alpha$ adapté à $X$ dans $N_X$ supposé connexe.
  On construit $\widehat F$ comme le relèvement de $F$ envoyant l'origine $\widehat z$ de $\widehat\alpha$ sur son extrémité.
  Considérons un automorphisme de revêtement $\widehat T$ et montrons l'égalité $\widehat T \circ \widehat F = \widehat F \circ \widehat T$.
  Choisissons un chemin $\widehat\gamma \subset \widehat N_X$ de $\widehat z$ à $\widehat T(\widehat z)$.
  Il se projette sur un lacet $\gamma \subset N_X$ qui est, par hypothèse, homotope au lacet $\alpha.F(\gamma).\alpha^{-}$.
  Cela signifie que les relèvements issus de $\widehat z$ des chemins $\alpha.F(\gamma)$ et $\gamma.\alpha$ ont même extrémité.
  Or le lacet $\alpha$ se relève à partir du point $\widehat z$ en le chemin $\widehat\alpha$ d'extrémité $\widehat F(\widehat z)$.
  À partir de ce point, $F(\gamma)$ se relève en $\widehat F(\widehat\gamma)$ d'extrémité $\widehat F \circ \widehat T(\widehat z)$.
  Par ailleurs, $\gamma$ se relève à partir de $\widehat z$ en $\widehat\gamma$ d'extrémité $\widehat T(\widehat z)$.
  Le relèvement de $\alpha$ issu de ce point est $\widehat T(\widehat\alpha)$ d'extrémité $\widehat T \circ \widehat F(\widehat z)$.
  On en déduit $\widehat T \circ \widehat F(\widehat z) = \widehat F \circ \widehat T(\widehat z)$ donc $\widehat T \circ \widehat F = \widehat F \circ \widehat T$. \qedhere
 \end{enumerate}
\end{proof}

\begin{Rem} \label{Rem:AutoExist}
 Soient $F \in \Homeo_*(M)$ et $X \subset \Fix(F)$ un ensemble fermé.
 La proposition précédente montre l'équivalence des propriétés suivantes :
 \begin{enumerate}
  \item il existe un relèvement $(X,\widehat F_X) \in \mc R$
  \item pour tout $z \in N_X$, il existe un chemin $\alpha_z$ adapté à $X$ d'origine $z$ ;
  \item il existe un chemin adapté à $X$ dans chaque composante connexe de $N_X$ ;
 \end{enumerate}
\end{Rem}

Lorsque les composantes connexes d'une surface $N$ ne sont pas homéomorphes au tore ou à l'anneau, le groupe des automorphismes de revêtement est de centre trivial.
On peut alors faire la remarque suivante.

\begin{Rem} \label{Rem:UniChAdapt}
  Soit $N$ une surface de revêtement universel $\widehat\pi : \widehat N \rightarrow N$, dont le groupe $\widehat G$ des automorphismes de revêtement est de centre trivial et soit $F$ un homéomorphisme de $N$.
  Alors deux chemins $\alpha$ et $\beta$ de même origine $z \in N_X$ et adaptés à $X$ sont homotopes à extrémités fixées dans $N_X$.
\end{Rem}

\begin{proof}
 En effet, d'après la proposition qui précède, chacun de ces deux chemins est associé à un relèvement de $F_{|N_X}$ qui commute avec les automorphismes du revêtement universel $\widehat\pi_X : \widehat N_X \rightarrow N_X$.
 Or, d'après la remarque \ref{Rem:UniCommu}, un tel relèvement $\widehat F$ est unique.
 Ainsi, $\alpha$ et $\beta$ se relèvent à $\widehat N_X$ en deux chemins d'origine un relevé $\widehat z$ de $z$ et de même extrémité $\widehat F(\widehat z)$.
 En d'autres termes, les chemins $\alpha$ et $\beta$ sont homotopes à extrémités fixées dans $N_X$.
\end{proof}

\subsection{Ensemble non enlacé}
\label{ssec:defenlac}

Afin de mieux comprendre les notions introduites ci-dessus, nous allons maintenant définir l'enlacement à partir de trois propriétés dont nous allons montrer
qu'elles sont équivalentes à l'aide du théorème d'Epstein \cite{Epstein66}.
Cette définition et cette équivalence permettent de mieux comprendre les notions utilisées dans l'article, mais elles ne seront pas utilisées sous cette forme dans la suite.
Nous avons préféré travailler avec une relation d'ordre qui sera introduite dans le paragraphe suivant.

\begin{PropDef} \label{PropDef:123}
 Soient $F \in \Homeo_*(M)$ et $X \subset \Fix(F)$ un ensemble fermé.
 On dit que $X$ est \textit{non enlacé} s'il vérifie l'une des propriétés équivalentes suivantes :
 \begin{enumerate}
  \item[\Pro1] il existe une isotopie restreinte $(X,I) \in \mc I$ ;
  \item[\Pro2] il existe un relèvement $(X,\widehat F) \in \mc R$ ;
  \item[\Pro3] tout lacet $\Gamma \subset N_X$ est librement homotope dans $N_X$ à son image $F(\Gamma)$.
 \end{enumerate}
\end{PropDef}

\begin{proof}[Démonstration de l'équivalence des propriétés \Pro1, \Pro2 et \Pro3]
 Nous ne démontrons pas que la propriété $\Pro3$ entraîne la propriété $\Pro1$.
 Ce résultat provient d'un théorème difficile dû à D.B.A. Epstein.
 En effet, $F$ préserve les composantes de $\partial (M \setminus X)$.
 D'après un théorème de Whitehead (\cite{Spanier89}, Corollary 24 page 405), la propriété $\Pro3$ entraîne que l'identité et $F$ sont homotopes par une homotopie de paires sur $(M\setminus X,\partial(M\setminus X))$.
 D'après un théorème d'Epstein (\cite{Epstein66}, Theorem 6.3), il existe donc une isotopie entre l'identité et $F$.

 Démontrons les deux autres implications.
  Si $X$ vérifie la propriété $\Pro1$, alors il existe une isotopie restreinte $(X,I) \in \mc I$.
 On a vu que l'on pouvait lui associer un relèvement $(X,\widehat F) \in \mc R$ (voir ce qui précède la définition \ref{Def:RelAsso}).
 Ainsi, $X$ vérifie la propriété~$\Pro2$.
 
 Supposons maintenant que $X$ vérifie la propriété $\Pro2$ et montrons qu'il vérifie alors la propriété~$\Pro3$.
 Considérons un lacet $\Gamma \subset N_X$ basé en $z$.
 Soit $\alpha$ un chemin associé à $\widehat F$ issu de $z$ ; il est donc adapté à $X$ d'après la proposition \ref{Prop:PlusPrecis}.
 Le lacet $\Gamma$ est donc homotope à  $\alpha.F(\Gamma).\alpha^{-}$ qui est lui-même librement homotope à $F(\Gamma)$.
 L'ensemble $X$ vérifie donc la propriété \Pro3.
\end{proof}

\begin{Rem} \label{Rem:UniCommu2}
 Soit $X$ un ensemble non enlacé.
 D'après la remarque \ref{Rem:UniCommu}, lorsque $M \setminus X$ n'est pas homéomorphe à l'anneau ou au tore, le relèvement $(X, \widehat F)$ donné par la propriété $\Pro2$ est unique.
 C'est donc aussi le relèvement associé à l'isotopie restreinte $(X,I)$ donné par la propriété $\Pro1$.
 
 En revanche, dans le cas de l'anneau ou du tore, il n'y a pas unicité du relèvement $(X, \widehat F) \in \mc R$.
 Cependant, étant donné un relèvement $(X, \widehat F) \in \mc R$, on peut toujours trouver une isotopie restreinte $(X,I) \in \mc I$ tel que $(X, \widehat F)$ soit associé à $(X,I)$.
 En effet, le temps un d'une isotopie restreinte $(X,I)$ donnée par $\Pro1$ diffère de $(X, \widehat F)$ d'un automorphisme de revêtement.
 Ce dernier est associé à une isotopie entre l'identité et elle-même réalisant la rotation correspondante de l'anneau ou du tore.
 En composant l'inverse de cette isotopie avec $(X,I)$, on obtient l'isotopie restreinte cherchée.
\end{Rem}

\begin{Qu} \label{Qu:NEF}
 Comme on l'a évoqué dans l'introduction, on peut définir une notion de non enlacement \textit{a priori} plus forte que la précédente de la façon suivante.
 On dit que $X$ est \textit{fortement non enlacé} s'il existe une isotopie $I$ entre l'identité $\id_M$ et $F$ vérifiant $X \subset \Fix(I)$.
 Nous ne savons pas si cette notion est équivalente à la précédente et ne l'utiliserons pas dans cet article.
\end{Qu}

\begin{Rem} \label{Rem:NEF}
 Dans le cas particulier où $X$ est totalement discontinu, et notamment si $X$ est fini, il est facile de répondre à la question précédente.
 En effet, si $X$ est non enlacé, on peut trouver une isotopie restreinte $(X, I) \in \mc I$.
 Cette isotopie se prolonge en une isotopie sur $M$ fixant les points de $X$ et $X$ est donc fortement non enlacé.
\end{Rem}

\section{Présentation de la preuve} \label{sec:pres}

\subsection{Relation d'ordre sur les relèvements}

Notre objectif est la démonstration du théorème \ref{Th:princ}.
Dans cette optique, nous recherchons un relèvement $(X,\widehat F_X) \in \mc R$ avec $\widehat F_X$ sans point fixe.
Pour cela, nous allons utiliser une relation d'ordre choisie de façon à ce que les relèvements maximaux conviennent (proposition \ref{Prop:MaximaFaible}).
Introduisons la relation ; nous montrerons ensuite (proposition \ref{Prop:RO}) qu'il s'agit bien d'une relation d'ordre.

\begin{Def} \label{Def:RO3}
 On définit la relation suivante sur l'ensemble $\mc R$ ; on note
 $\bigl(X,\widehat F_X\bigr) \pprec \bigl(Y,\widehat F_Y\bigr)$ si :
 \begin{enumerate}
 \item on a 
   $X \subset Y \subset \Bigl(X \cup \widehat\pi_X\bigl(\Fix(\widehat F_X)\bigr)\Bigr)$ ;
 \item tout chemin de $N_Y$ associé à $\widehat F_Y$ est également associé à $\widehat F_X$.
\end{enumerate}
\end{Def}

\begin{Rem}
 L'ensemble $\widehat\pi_X(\Fix(\widehat F_X))$ est fermé.
\end{Rem}

\begin{Rem}
 Si $\widehat F_X$ est sans point fixe, alors $(X, \widehat F_X)$ est maximal.
 En effet, l'assertion (1) entraîne alors $X=Y$.
\end{Rem}

\begin{Rem} \label{Rem:R01}
 La relation $(X,\widehat F_X) \pprec (Y,\widehat F_Y)$ entraîne l'inclusion $\widehat\pi_Y\bigl(\Fix(\widehat F_Y)\bigr) \subset \widehat\pi_X\bigl(\Fix(\widehat F_X)\bigr)$.
 En effet si $z$ appartient à $\widehat\pi_Y\bigl(\Fix(\widehat F_Y)\bigr)$, alors le lacet trivial basé en $z$ est associé à $\widehat F_Y$ donc à $\widehat F_X$.
\end{Rem}
 
\begin{Rem} \label{Rem:R02}
 Si $(X,\widehat F_X) \pprec (Y,\widehat F_Y)$ et $X=Y$, alors $(X,\widehat F_X) = (Y,\widehat F_Y)$.
 En effet, par hypothèse, tout chemin de $N_X = N_Y$ associé à $\widehat F_Y$ est aussi associé à $\widehat F_X$ ce qui entraîne l'égalité $\widehat F_X = \widehat F_Y$.
\end{Rem}

\begin{Rem} \label{Rem:R03}
 On peut remplacer dans la définition la deuxième propriété par la propriété \textit{a priori} plus faible :
\begin{quotation}
«~il existe dans chaque composante connexe de $N_Y$ un chemin associé à $\widehat F_Y$ et à $\widehat F_X$~».
\end{quotation}
 En effet, supposons l'existence d'un chemin $\alpha$ associé à $\widehat F_Y$ et à $\widehat F_X$ dans une composante connexe $U$ de $N_X$.
 Montrons qu'alors tout chemin de $U$ associé à $\widehat F_Y$ est aussi associé à $\widehat F_X$.
 Soit $\alpha' \subset U$ associé à $\widehat F_Y$.
 Relevons les chemins $\alpha$ et $\alpha'$ en $\widehat\alpha \subset \widehat N_Y$ et $\widehat\alpha' \subset \widehat N_Y$.
 Choisissons un chemin $\widehat\gamma$ joignant l'origine de $\widehat\alpha$ à celle de $\widehat\alpha'$.
 Le chemin $\widehat F_Y(\widehat\gamma)$ joint l'extrémité de $\widehat\alpha$ à celle de $\widehat\alpha'$.
 Comme les composantes connexes de $\widehat N_Y$ sont simplement connexes, on peut donc construire une famille continue $(\widehat\alpha_t)_{t \in [0,1]}$ de chemins telle que $\widehat\alpha_0 = \widehat\alpha$, $\widehat\alpha_1 = \widehat\alpha'$ et pour tout $t \in [0,1]$, $\widehat\alpha_t$ ait pour origine $\widehat\gamma(t)$ et pour extrémité $\widehat F_Y\bigl(\widehat\gamma(t)\bigr)$.
 Par projection, on obtient alors une famille continue $(\alpha_t)_{t \in [0,1]}$ de chemins.
 Comme par hypothèse $\alpha_0 = \alpha$ est associé à $\widehat F_X$, tous les chemins $\alpha_t$ sont associés à $\widehat F_X$.
 En particulier, le chemin $\alpha_1 = \alpha'$ est associé à $\widehat F_X$.
\end{Rem}

\begin{Prop} \label{Prop:RO}
 La relation $\pprec$ est une relation d'ordre sur $\mc R$.
\end{Prop}

\begin{proof}
 Il est clair que la relation $\pprec$ est réflexive.
 Montrons qu'elle est antisymétrique.
 Supposons $(X,\widehat F_X) \pprec (Y,\widehat F_Y)$ et $(Y,\widehat F_Y) \pprec (X,\widehat F_X)$.
 On en déduit $X = Y$ puis $(X,\widehat F_X) = (Y,\widehat F_Y)$ d'après la remarque \ref{Rem:R02}.
 
 Enfin, montrons que $\pprec$ est transitive.
 Supposons $(X,\widehat F_X) \pprec (Y,\widehat F_Y)$ et $(Y,\widehat F_Y) \pprec (Z,\widehat F_Z)$.
 On a donc :
 \begin{gather*}
  X \subset Y \subset \Bigl(X \cup \widehat\pi_X\bigl(\Fix(\widehat F_X)\bigr)\Bigr)
  \qt{et}
  Y \subset Z \subset \Bigl(Y \cup \widehat\pi_Y\bigl(\Fix(\widehat F_Y)\bigr)\Bigr)\\
  \text{puis}\qquad
  X \subset Z
  \subset \Bigl(Y \cup \widehat\pi_Y\bigl(\Fix(\widehat F_Y)\bigr)\Bigr)
  \subset \Bigl(X \cup \widehat\pi_X\bigl(\Fix(\widehat F_X)\bigr)\Bigr),
 \end{gather*}
 en utilisant la remarque \ref{Rem:R01} pour obtenir la dernière inclusion.
 La propriété $2$ est immédiatement vérifiée.
 \end{proof}

Avant de poursuivre, revenons sur la remarque \ref{Rem:UniCommu2}.
Lorsque certaines composantes connexes de $N_X$ sont homéomorphes à l'anneau ou au tore, il n'y a pas unicité, s'il existe, du relèvement $(X,\widehat F_X) \in \mc R$.
Nous avons expliqué dans l'introduction que cela justifiait la présence de la troisième condition dans les conclusions du théorème \ref{Th:princ}.
Un exemple d'une telle situation est donné par l'homéomorphisme $F : \C \to \C$ défini par :
\[
 F(z) = e^{8i\arctan(|z|)} z.
\]
L'ensemble $X = \Fix(F)$ est la réunion du cercle unité et du point $0$.
On cherche un relèvement $(X,\widehat F_X) \in \mc R$ : choisissons celui qui est associé à la restriction à $N_X$ de l'isotopie $I$ définie par :
\[
 \phi : \fonction{[0,1] \times \C}\C{(t,z)}{e^{8it\arctan(|z|)} z}.
\]
Ce relèvement, simple à construire, vérifie les deux premières conditions du théorème \ref{Th:princ} mais pas la troisième.

Pour le voir, considérons par exemple le sous-ensemble $Y =  \{0,1\}$ de $X$.
Pour satisfaire la troisième condition, il faudrait d'abord trouver un relèvement $(Y,\widehat F_Y) \in \mc R$.
D'après la remarque \ref{Rem:UniCommu2}, $N_Y$ n'étant homéomorphe ni à l'anneau ni au tore, un tel relèvement est unique.
C'est donc nécessairement celui qui est associé à la restriction à $N_Y$ de l'isotopie $J$ définie par :
\[
 \psi : \fonction{[0,1] \times \C}\C{(t,z)}{e^{it(8\arctan(|z|)-2\pi)} z}.
\]
Mais les chemins $t \mapsto \phi(t,\sqrt3) = e^{\frac{8it\pi}3} \sqrt3$ et $t \mapsto \psi(t,\sqrt3) = e^{\frac{2it\pi}3} \sqrt3$ ne sont pas homotopes à extrémités fixées dans $N_Y$.
Le relèvement $(X, \widehat F_X)$ ne vérifie donc pas la fin de la troisième condition.

Cependant, l'isotopie $J$ est mieux choisie que la précédente car elle vérifie $\Fix(J) = X$, ce qui montre d'ailleurs que $X$ est fortement non enlacé.
À la restriction de $J$ à $N_X$ est associé un relèvement $(X, \widehat F'_X)$.
C'est celui que l'on obtient à partir de la notion de non enlacement fort et le seul à vérifier la troisième condition du théorème \ref{Th:princ}. On peut caractériser ce relèvement : c'est le seul qui appartient à l'ensemble $\mc R'$ suivant.

\begin{Def}
 On note $\mc R'$ l'ensemble des couples $(X,\widehat F_X) \in \mc R$ tels que pour tout sous-ensemble fermé $Y \subset X$, il existe $(Y,\widehat F_Y) \in \mc R$ vérifiant $(Y,\widehat F_Y) \pprec (X,\widehat F_X)$.
 La relation $\pprec$ définit un ordre sur~$\mc R'$.
\end{Def}

Avant de poursuivre, et pour clarifier les idées, vérifions que l'unicité du relèvement, qui a motivé l'introduction de $\mc R'$, est bien une propriété vérifiée dans $\mc R'$.

\begin{Prop} \label{Prop:UniRR}
 Soit $X \subset \Fix(F)$ un sous-ensemble fermé non enlacé.
On exclut les cas particuliers suivants :
\begin{itemize}
 \item $M$ est l'anneau ou le tore et $X$ est l'ensemble vide ;
 \item $M$ est le disque et $X$ un singleton ;
 \item $M$ est la sphère et $X$ est de cardinal $2$.
\end{itemize}
Alors il existe au plus un relèvement $(X, \widehat F_X) \in \mc R'$.
\end{Prop}

La démonstration de ce résultat repose sur le lemme suivant :

\begin{Lemme} \label{Le:T1}
Soit $X \subset M$ une partie fermée et $\Gamma \subset N_X$ un lacet.
Alors $\Gamma$ est homotope à zéro dans $N_X$ si et seulement si $\Gamma$ est homotope à zéro dans $N_Z$ pour toute partie finie $Z \subset X$.
\end{Lemme}

\begin{proof}[Démonstration du lemme \ref{Le:T1}]
Si $\Gamma$ est homotope à zéro dans $N_X$, alors pour toute partie finie $Z \subset X$ on a $N_X \subset N_Z$ donc $\Gamma$ est homotope à zéro dans $N_Z$.

Réciproquement, supposons que pour toute partie finie $Z \subset X$, le lacet $\Gamma$ est homotope à zéro dans $N_Z$.
Considérons une surface à bord compacte connexe $K_\Gamma \subset N_X$ contenant $\Gamma$.
Comme $M$ est connexe, toute composante connexe de $M \setminus K_\Gamma$ rencontre une composante de la frontière $\partial K_\Gamma$ de $K_\Gamma$.
Par compacité de $K_\Gamma$, les composantes connexes de $M \setminus K_\Gamma$ sont donc en nombre fini.
Pour chaque composante connexe $C$ de $M \setminus K_\Gamma$ rencontrant $X$, on choisit un point $z_C \in C \cap X$.
On note $Z$ l'ensemble fini obtenu par réunion de ces points.
Enfin, on note $N$ la réunion de $K_\Gamma$ et des composantes connexes de $M \setminus K_\Gamma$ qui ne rencontrent pas $C$.

Considérons le revêtement universel $\widehat\pi : \widehat N_Z \to N_Z$ de $N_Z$ et une composante connexe $\widehat N \subset \widehat N_Z$ de la préimage de $N$ par $\widehat\pi$.
Le lacet $\Gamma$ se relève à $\widehat N$ en un chemin $\widehat\Gamma$ issu de $\widehat z$ qui est un lacet puisque, par hypothèse, $\Gamma$ est homotope à zéro dans $N_Z$.
Par ailleurs, remarquons que $\widehat N$ est simplement connexe car s'il existait une composante connexe bornée de $\widehat N_Z \setminus \widehat N$, elle se projetterait sur une composante connexe de $M \setminus N$ ne rencontrant pas $Z$ ce qui contredirait la construction de $N$ et $Z$.
Le lacet $\widehat\Gamma$ est donc homotope à zéro dans $\widehat N$ et son projeté $\Gamma$ est donc homotope à zéro dans $N \subset N_X$ donc dans $N_X$.
\end{proof}

\begin{proof}[Démonstration de la proposition \ref{Prop:UniRR}]
L'unicité de $(X, \widehat F_X)$ est immédiate lorsque le groupe des automorphismes du revêtement universel de $N_X$ est de centre trivial.
Lorsque $X$ est fini, cette propriété est toujours vérifiée ici compte tenu des cas particuliers que nous avons exclus.

Supposons maintenant $X$ infini.
Soient deux relèvements $(X, \widehat F_X) \in \mc R'$ et $(X, \widehat G_X) \in \mc R'$.
Choisissons un point $z \in N_X$ et deux chemins $\alpha$ et $\beta$ issus de $z$, le premier associé à $\widehat F_X$ et le second à $\widehat G_X$.
Considérons maintenant le lacet $\Gamma = \alpha.\beta^-$.

Pour tout sous-ensemble fini $Z \subset X$ de cardinal au moins $3$, par définition de $\mc R'$, on peut trouver $(Z, \widehat F_Z) \in \mc R$ vérifiant $(Z, \widehat F_Z) \pprec (X, \widehat F_X)$ et $(Z, \widehat G_Z) \in \mc R$ vérifiant $(Z, \widehat G_Z) \pprec (X, \widehat G_X)$.
Mais puisque $Z$ est fini de cardinal au moins $3$, le groupe des automorphismes du revêtement universel de $N_Z$ est de centre trivial et $\widehat F_Z$ est donc unique d'où $\widehat F_Z = \widehat G_Z$.
Les chemins $\alpha$ et $\beta$ sont donc tous deux associés à $\widehat F_Z$ donc $\Gamma$ est homotope à zéro dans $N_Z$.

Si $Z \subset X$ est un sous-ensemble fini de cardinal inférieur ou égal à $2$, on peut trouver un ensemble $Z'$ fini de cardinal $3$ vérifiant $Z \subset Z' \subset X$.
D'après ce qui précède, $\Gamma$ est homotope à zéro dans $N_{Z'}$ donc dans $N_Z$ comme précédemment.

Nous pouvons donc utiliser le lemme \ref{Le:T1} et en déduire que $\Gamma$ est homotope à zéro dans $N_X$ ce qui montre que $\alpha$ et $\beta$ sont homotopes dans $N_X$ donc $\widehat F_X = \widehat G_X$.
\end{proof}

\subsection{Preuve du théorème \ref{Th:princ}}

Nous démontrerons dans la partie \ref{sec:induc} les résultats suivants concernant l'existence de relèvements maximaux.

\begin{Prop} \label{Prop:existmax}
 Soient $F \in \Homeo_*(M)$ et $(X,\widehat F_X) \in \mc R$.
 Alors il existe $(Y,\widehat F_Y) \in \mc R$ maximal vérifiant $(X,\widehat F_X) \pprec (Y,\widehat F_Y)$.
\end{Prop}

\PropBis{existmax}
\begin{Propexistmax}
 Soient $F \in \Homeo_*(M)$ et $(X,\widehat F_X) \in \mc R'$.
 Alors il existe $(Y,\widehat F_Y) \in \mc R'$ maximal dans $\mc R'$ vérifiant $(X,\widehat F_X) \pprec (Y,\widehat F_Y)$.
\end{Propexistmax}

Pour prouver le théorème \ref{Th:princ}, il reste à montrer que les relèvements maximaux permettent l'utilisation du théorème \ref{Th:Brouwer} c'est-à-dire sont sans point fixe.

\begin{Prop} \label{Prop:MaximaFaible}
 Soit $(X, \widehat F_X) \in \mc R$.
 S'il existe un point $y \in \Fix(F) \setminus X$ dont les relevés à $\widehat N_X$ sont des points fixes de $\widehat F_X$, alors $(X,\widehat F_X)$ n'est pas maximal dans $(\mc R,\pprec)$.
 
 Plus précisément, si l'on note $Y = X \cup \{y\}$, il existe $\bigl(Y,\widehat F_Y\bigr) \in \mc R$ vérifiant $\bigl(X,\widehat F_X\bigr) \pprec \bigl(Y,\widehat F_Y\bigr)$.
\end{Prop}

\PropBis{MaximaFaible}

\begin{PropMaximaFaible}
 Soit $(X, \widehat F_X) \in \mc R'$.
 S'il existe un point $y \in \Fix(F) \setminus X$ dont les relevés à $\widehat N_X$ sont des points fixes de $\widehat F_X$, alors $(X,\widehat F_X)$ n'est pas maximal dans $(\mc R',\pprec)$.
 
 Plus précisément, si l'on note $Y = X \cup \{y\}$, il existe $\bigl(Y,\widehat F_Y\bigr) \in \mc R'$ vérifiant $\bigl(X,\widehat F_X\bigr) \pprec \bigl(Y,\widehat F_Y\bigr)$.
\end{PropMaximaFaible}

\begin{Rem}
 Comme $\widehat F_X$ commute avec les automorphismes du revêtement universel $\widehat\pi_X : \widehat N_X \rightarrow N_X$, s'il fixe un relevé à $\widehat N_X$ d'un point $y \in \Fix(F) \setminus X$, il les fixe tous.
\end{Rem}

\begin{Rem}
 Considérons un relèvement $(X, \widehat F_X) \in \mc R'$.
 D'après la proposition \ref{Prop:MaximaFaible} (respectivement \refbis{Prop:MaximaFaible}), le relèvement $(X, \widehat F_X)$ est maximal dans $\mc R$ (resp. dans $\mc R'$) si et seulement si $\widehat F_X$ est sans point fixe.
 Ainsi $(X, \widehat F_X)$ est maximal dans $\mc R$ si et seulement s'il est maximal dans $\mc R'$.
\end{Rem}

Avant de démontrer les propositions \ref{Prop:MaximaFaible} et \refbis{Prop:MaximaFaible} dans le paragraphe suivant, nous allons prouver le théorème \ref{Th:princ}.
Nous nous appuyons pour cela sur les propositions \refbis{Prop:MaximaFaible} démontrée ci-après et \refbis{Prop:existmax} démontrée dans la partie \ref{sec:induc}.
Signalons que pour obtenir la troisième condition dans le théorème \ref{Th:princ}, nous utilisons ces propositions et non pas les propositions \ref{Prop:MaximaFaible} et \ref{Prop:existmax} ; nous donnerons cependant une preuve de l'ensemble de ces résultats.

\begin{proof}[Démonstration du théorème \ref{Th:princ}]
 Considérons $(\emptyset,\widetilde F) \in \mc R'$ où $\widetilde F$ est le relèvement associé à $I$ de $F$ au revêtement universel $\widetilde M$ de $M$.
 D'après la proposition \refbis{Prop:existmax}, il existe $(X,\widehat F_X) \in \mc R'$ maximal et vérifiant $(\emptyset,\widetilde F) \pprec (X,\widehat F_X)$, ce qui signifie que tout chemin associé à $\widehat F_X$ est aussi associé à $\widetilde F$.
 De plus, d'après la proposition \refbis{Prop:MaximaFaible}, $\widehat F_X$ n'a pas de point fixe.
 Enfin, la troisième assertion du théorème repose sur la définition de $\mc R'$.
 En particulier, l'affirmation (a) est la traduction de la propriété $X \subset \left( Y \cup \widehat\pi_Y(\Fix(\widehat F_Y)) \right)$.
\end{proof}

\subsection{Éléments maximaux de $(\mc R, \pprec)$ et $(\mc R', \pprec)$}

L'objet de ce paragraphe est la preuve des propositions \ref{Prop:MaximaFaible} et \refbis{Prop:MaximaFaible}.
Celle-ci s'appuie sur la proposition \ref{PropTh:XyNonEnlace} qui découle du lemme d'Alexander mais dont nous donnerons une démonstration complète dans la partie \ref{sec:contrac}.
Il s'agit d'une formulation de la proposition \ref{Prop:MaximaFaible} dans le langage des isotopies restreintes.
Pour déduire les propositions \ref{Prop:MaximaFaible} et \refbis{Prop:MaximaFaible} de la proposition \ref{PropTh:XyNonEnlace}, nous allons utiliser le théorème d'Epstein (proposition \ref{PropDef:123}, $\Pro3 \Rightarrow \Pro1$).
Signalons que notre preuve du théorème \ref{Th:princ} fait appel au théorème d'Epstein uniquement pour ces preuves des propositions \ref{Prop:MaximaFaible} et \refbis{Prop:MaximaFaible}.

\begin{Prop} \label{PropTh:XyNonEnlace}
 Soient $F \in \Homeo_*(M)$ et
 $(X,I) \in \mc I$ une isotopie restreinte.
 Soit $y \in \Cont(I) \setminus \Fix(I)$.
 Alors il existe une isotopie restreinte $(X,I') \in \mc I$ fixant $y$ avec $I'$ homotope à $I$ relativement au complémentaire d'un compact de $N_X$.
\end{Prop}

\begin{proof}[Démonstration de la proposition \ref{Prop:MaximaFaible}]
 Remarquons d'abord qu'il suffit de s'intéresser à la composante connexe $N$ de $N_X$ qui contient le point $y$.
 En effet, les autres composantes connexes sont inchangées lorsque l'on enlève $y$ et leur revêtement universel est également inchangé.
 Pour simplifier les notations, nous supposons donc $N_X$ connexe.
 
 Utilisons le théorème d'Epstein.
 D'après la proposition \ref{PropDef:123} et la remarque \ref{Rem:UniCommu2}, il existe une isotopie restreinte $(X,I_X)$ telle que $(X,\widehat F_X)$ soit associé à $(X,I_X)$.
 D'après la proposition \ref{PropTh:XyNonEnlace}, on peut modifier $(X,I_X)$ pour obtenir une isotopie $(X,J_Y)$ qui fixe $y$.
 Considérons le relèvement $(X,\widehat F_{J_Y}) \in \mc R$ associé à l'isotopie restreinte $(X,J_Y)$.
 Comme $\widehat F_{J_Y}$ est un relèvement de $F_{|N_X}$ qui fixe les relevés de $y$, c'est $\widehat F_X$.

 Notons $I_Y$ la restriction de $J_Y$ à $N_Y$ et
 considérons le relèvement $(Y, \widehat F_Y) \in \mc R$ associé à l'isotopie restreinte $(Y,I_Y)$.
 Il reste à vérifier que l'on a bien $(X, \widehat F_X) \pprec (Y, \widehat F_Y)$ :
 \begin{enumerate}
  \item
  par hypothèse, $Y = X \cup \{y\}$ vérifie $X \subset Y \subset
  \Bigl( X \cup \widehat\pi_X \bigl( \Fix(\widehat F_X) \bigr) \Bigr)$ ;
  \item
  soit $z \in N_Y$.
  Considérons le chemin $I_Y(z)$ qui n'est autre que $J_Y(z)$.
  Il est donc associé à $\widehat F_Y$ mais aussi à $\widehat F_X$, ce qui achève la démonstration d'après la remarque \ref{Rem:R03}.\qedhere
 \end{enumerate}
\end{proof}

Nous pouvons également démontrer facilement la proposition \refbis{Prop:MaximaFaible} dans le cas particulier où $X$ est totalement discontinu.
Cette situation est relativement simple car $Y$ est alors fortement non enlacé (voir la remarque~\ref{Rem:NEF}).

\begin{proof}[Démonstration de la proposition \refbis{Prop:MaximaFaible} lorsque $X$ est totalement discontinu]
D'après la proposition \ref{PropDef:123}, il existe une isotopie restreinte $(X, I_X) \in \mc I$ associée au relèvement $(X, \widehat F_X)$.
Comme les points de $X$ sont isolés, cette isotopie se prolonge en une isotopie $I$ définie sur $M$ et fixant les points de~$X$.
Par hypothèse, les relevés de $y$ à $\widehat N_X$ sont des points fixes de $\widehat F_X$.
D'après la proposition \ref{PropTh:XyNonEnlace} on peut modifier $I$ en une isotopie $J$ qui fixe les points de $Y$.
À la restriction à $N_{Y}$ de l'isotopie $J$ est associé un relèvement $(Y, \widehat F_Y) \in \mc R$.

Montrons que $(Y, \widehat F_Y)$ appartient à $\mc R'$.
En effet, pour tout sous-ensemble $Z \subset Y$, la restriction de $J$ à $N_Z$ permet de définir un relèvement \mbox{$(Z, \widehat F_Z) \in \mc R$} vérifiant $(Z, \widehat F_Z) \pprec (Y, \widehat F_Y)$ puisque pour tout point $z \in N_Y$, le chemin $J(z)$ est associé à $\widehat F_Y$ et à $\widehat F_Z$.
\end{proof}

Intéressons-nous maintenant à la proposition \refbis{Prop:MaximaFaible} dans le cas où $X$ n'est pas totalement discontinu.
Pour la démontrer, nous allons utiliser les lemmes suivants :

\begin{Lemme} \label{Le:T2}
Soient $(Z, \widehat F_Z)$ et $(Y, \widehat F_Y)$ deux éléments de $\mc R$ vérifiant $Z \subset Y \subset \Bigl(Z \cup \,\widehat\pi_Z\bigl(\Fix(\widehat F_Z)\bigr)\Bigr)$.
On suppose de plus que l'ensemble $Z$ contient un point $y$ isolé dans $Y$ et que les composantes connexes de $N_Z$ et de $N_Y$ qui contiennent $y$ dans leur adhérence ne sont homéomorphes ni à l'anneau ouvert, ni au disque.

Alors, tout anneau $A \subset N_Y$ entourant $y$ (c'est-à-dire tel que $A \cup \{y\}$ soit un disque) contient des chemins associés à $\widehat F_Y$.
Ces chemins sont de plus associés à $\widehat F_Z$.
Dans le cas particulier où $N_Y$ est connexe, cela signifie que l'on a $(Z, \widehat F_Z) \pprec (Y, \widehat F_Y)$.
\end{Lemme}

\begin{proof}
D'après la proposition \ref{PropDef:123}, on sait qu'il existe une isotopie $(Z,I_Z)$  telle que $(Z, \widehat F_Z)$ soit associé à $(Z,I_Z)$ et une isotopie $(Y,I_Y)$ telle que $(Y, \widehat F_Y)$ soit associé à $(Y,I_Y)$.
Comme $y$ est un point de $Z$ isolé dans $Y$ donc dans $Z$, le lacet $\beta = \bigl(I_Y.(I_Z)^- \bigr)(w)$ est contenu dans $A$ si l'on choisit $w$ suffisamment proche de $y$.
Choisissons un tel point $w$ 
puis un lacet $\Gamma \subset N_Y$ basé en $w$ et non homotope à zéro dans $N_Y \cup \{y\}$.
Si l'on suppose que $N_Y$ n'est homéomorphe ni à l'anneau ni au disque, un tel choix est possible puisque $N_Y \cup \{y\}$ ne peut alors être homéomorphe ni au disque ni à la sphère.

On note $C_Z$ la composante connexe de $N_Z$ dont l'adhérence contient $y$ ; elle contient également $w$.
On note également $\widehat\pi : \widehat C_Z \to C_Z$ le revêtement universel de $C_Z$ et $\widehat G$ le groupe des automorphismes de revêtement.
Choisissons un relèvement $\widehat w \in \widehat C_Z$ de $w$ et notons $\widehat\beta$ et $\widehat\Gamma$ les relèvements de $\beta$ et $\Gamma$ issus de $\widehat w$.
L'extrémité de $\widehat\beta$ relève $w$ et s'écrit donc sous la forme $\widehat T(\widehat w)$ avec $\widehat T \in \widehat G$.
De même, l'extrémité de $\widehat\Gamma$ s'écrit $\widehat U(\widehat w)$ avec $\widehat U \in \widehat G$.
Remarquons maintenant que $I_Y$ définit dans $N_Y$, donc dans $N_Z$, et plus précisément dans $C_Z$, une homotopie entre $\Gamma$ et $F(\Gamma)$.
L'isotopie $(I_Z)^-$ définit ensuite une homotopie entre $F(\Gamma)$ et $\Gamma$ toujours dans $C_Z$.
En relevant à $\widehat C_Z$ ces deux homotopies successives, on obtient une homotopie libre entre $\widehat\Gamma$ et $\widehat T(\widehat\Gamma)$.
Le long de cette homotopie, l'origine de $\widehat\Gamma$ décrit $\widehat\beta$ et son extrémité $\widehat U(\widehat w)$ décrit $\widehat U(\widehat\beta)$ d'extrémité $\widehat U \circ \widehat T(\widehat w)$.
Or l'extrémité de $\widehat T(\widehat\Gamma)$ est $\widehat T \circ \widehat U(\widehat w)$ et on obtient ainsi $\widehat U \circ \widehat T(\widehat w) = \widehat T \circ \widehat U(\widehat w)$ donc $\widehat U \circ \widehat T = \widehat T \circ \widehat U$.

Comme l'on suppose que $C_Z$ n'est homéomorphe ni au disque, ni à l'anneau, on en déduit que $\widehat U$ et $\widehat T$ appartiennent à un même sous-groupe monogène infini de $\widehat G$.
Il existe donc $(k,l) \in \Z^2 \setminus \{(0,0)\}$ tel que $\widehat T^k = \widehat U^l$ ce qui signifie que $\beta^k$ et $\Gamma^l$ sont librement homotopes dans $C_Z$ donc dans $C_Z \cup \{y\}$.
Or $\beta$ est homotope à zéro dans $C_Z \cup \{y\}$ puisque l'on a choisi $w$ de sorte que $\beta$ soit contenu dans l'anneau $A$ autour de $y$.
En revanche, $\Gamma^l$ n'est homotope à zéro dans $C_Z \cup \{y\}$ que si $l = 0$ (\cite{Epstein66}, Lemma 4.3) auquel cas $\widehat T^k$ est l'identité avec $k \neq 0$ donc $\widehat T$ est l'identité.
Ainsi le lacet $\beta$ est homotope à zéro dans $C_Z$.
On en déduit que $I_Y(w)$, qui est associé à $\widehat F_Y$, est homotope à extrémités fixées dans $C_Z$ à $I_Z(w)$ lui-même associé à $\widehat F_Z$.
Ainsi $I_Y(w)$ est associé à $\widehat F_Y$ et à $\widehat F_Z$.
Par connexité de l'anneau $A$, tout chemin contenu dans $A$ et associé à $\widehat F_Y$ sera également associé à $\widehat F_Z$ ce qui achève la démonstration. 
\end{proof}

Nous allons également utiliser un lemme topologique qui repose sur des arguments du type lemme d'Alexander :

\begin{Lemme}\label{Le:Alex}
Soit $D$ un disque ouvert contenu dans un ouvert $N$ d'une surface $M$.
Soient $y$ et $y'$ deux points de $D$.
Soient $\alpha$ et $\beta$ deux chemins contenus dans $N \setminus \overline D$, ayant même origine $z$ et même extrémité $z'$.

Si $\alpha$ et $\beta$ sont homotopes à extrémités fixées dans $N \setminus \{y'\}$, alors ils sont homotopes à extrémités fixées dans $N \setminus \{y\}$.
\end{Lemme}

\begin{proof}
Soit $\phi : [0,1]^2 \to N$ une homotopie entre $\alpha$ et $\beta$, à valeurs dans $N \setminus \{y'\}$ c'est-à-dire vérifiant :
\begin{itemize}
 \item pour tout $t \in [0,1]$, $\phi(0,t) = \alpha(t)$ et $\phi(1,t) = \beta(t)$ ;
 \item pour tout $s \in [0,1]$, $\phi(s,0) = z$ et $\phi(s,1) = z'$.
\end{itemize}
Comme $\overline D$ est fermé, et $\alpha$, $\beta$ sont contenus dans son complémentaire, il existe $\epsilon \in ]0,1[$ tel que pour tout $(s,t) \in ([0,\epsilon] \cup [1-\epsilon,1])\times[0,1]$ l'on ait $\phi(s,t) \notin \overline D$.
Par ailleurs, on peut construire une isotopie $(\psi_s)_{s \in [0,1]}$ entre l'identité $\id_N$ et elle-même, à support dans $D$, vérifiant :
\[
\forall s \in [\epsilon,1-\epsilon] \qquad 
\psi_s(y') = y
\]
Une construction plus détaillée de ce type d'application figure dans la partie~\ref{sec:contrac}.
L'application $(s,t) \mapsto \psi_s\bigl(\phi(s,t)\bigr)$ réalise une homotopie entre~$\alpha$ et~$\beta$ dans $N \setminus \{y\}$ car $\phi(s,t)$ ne vaut jamais $y'$ donc $\psi_s\bigl(\phi(s,t)\bigr)$ n'est jamais~$y$.
\end{proof}

Nous supposons ici que $X$ n'est pas totalement discontinu, et en particulier que $X$ est infini.
Dans la preuve de la proposition \ref{Prop:MaximaFaible}, nous avons en général construit le seul relèvement $(Y,\widehat F_Y)$ possible sauf dans un cas bien particulier : lorsque la composante $C_X$ de $N_X$ contenant $y$ est un disque.
En effet, dans ce cas, le groupe des automorphismes de revêtement de $C_Y = C_X \setminus \{y\}$ n'est pas de centre trivial (voir remarque \ref{Rem:UniCommu}).

Une première étape de la preuve va consister à choisir $\widehat F_Y$ convenablement lorsque $C_X$ est un disque ; ce relèvement $\widehat F_Y$ sera fourni par la proposition qui suit.

\begin{Prop} \label{Prop:T3}
Soient $\bigl(X, \widehat F_X\bigr) \in \mc R'$ et $y \in \Fix(F) \setminus X$ tel que les relevés de $y$ à $\widehat N_X$ soient des points fixes de $\widehat F_X$.
On suppose $X$ de cardinal supérieur ou égal à $2$, on note $C_X$ la composante connexe de $N_X$ contenant $y$ et $C_Y = C_X \setminus \{y\}$.

Alors, il existe un unique relèvement $(Y, \widehat F_Y) \in \mc R$ vérifiant :
\begin{enumerate}
 \item $\bigl(X, \widehat F_X\bigr) \pprec \bigl(Y, \widehat F_Y\bigr)$ ;
 \item pour tout sous-ensemble fini $Z_1 \subset X$ de cardinal supérieur ou égal à $2$, en notant $Z = Z_1 \cup\{y\}$ et $(Z, \widehat F_Z) \in \mc R$ l'unique relèvement correspondant, tout chemin de $C_Y$ associé à $\widehat F_Y$ est associé à $\widehat F_Z$.
\end{enumerate}
 \end{Prop}

\begin{proof}
Commençons par vérifier, pour tout sous-ensemble $Z_1 \subset X$ de cardinal supérieur ou égal à $2$, l'existence et l'unicité d'un relèvement $\bigl(Z, \widehat F_Z\bigr) \in \mc R$ où l'on a posé $Z = Z_1 \cup \{y\}$.
D'après l'hypothèse $\bigl(X, \widehat F_X\bigr) \in \mc R'$, il existe $\bigl(Z_1, \widehat F_{Z_1}\bigr) \in \mc R$ vérifiant $\bigl(Z_1, \widehat F_{Z_1}\bigr) \pprec \bigl(X, \widehat F_X\bigr)$.
 On en déduit, puisque les relevés de $y$ sont des points fixes de $\widehat F_X$, qu'ils sont également points fixes de $\widehat F_{Z_1}$ (voir remarque \ref{Rem:R01}) d'où l'existence de $\bigl(Z, \widehat F_Z\bigr) \in \mc R$ d'après la proposition \ref{Prop:MaximaFaible}.
L'unicité de $\bigl(Z, \widehat F_Z\bigr)$ provient quant à elle de la remarque \ref{Rem:UniCommu} car $Z$ est fini de cardinal supérieur ou égal à $3$.

Nous allons maintenant procéder en deux étapes.
Dans un premier temps, nous allons montrer que pour tout sous-ensemble fini $Z_1 \subset X$ de cardinal supérieur ou égal à~$2$, en posant $Z = Z_1 \cup \{y\}$, il existe un unique relèvement $\bigl(Y, \widehat F_Y\bigr) \in \mc R$ vérifiant $\bigl(X, \widehat F_X\bigr) \pprec \bigl(Y, \widehat F_Y\bigr)$ et 
tout chemin de $C_Y$ associé à $\widehat F_Y$ est associé à $\widehat F_Z$.
Dans un second temps, nous montrerons que ce relèvement $\bigl(Y, \widehat F_Y\bigr)$ ne dépend pas de $Z_1$ ; 
la démonstration sera alors achevée.

\textbf{Construction de $\bigl(Y, \widehat F_Y\bigr)$ pour un sous-ensemble fini $Z_1 \subset X$ fixé, de cardinal supérieur ou égal à $2$.}
Remarquons que l'on peut travailler indépendamment sur chaque composante connexe de $N_X$.
Sur le revêtement universel des composantes connexes de $N_X$ autres que $C_X$, la première condition impose de choisir $\widehat F_Y$ coïncidant avec $\widehat F_X$ autrement dit l'on définit $\widehat F_Y$ à partir des chemins associés à $\widehat F_X$ qui seront également associés à $\widehat F_Y$.

Intéressons-nous maintenant à $C_X$ et distinguons deux cas :
\begin{itemize}
 \item[\textbullet]
si $C_X$ n'est pas homéomorphe au disque, l'existence de $\bigl(Y,\widehat F_Y\bigr) \in \mc R$ vérifiant $\bigl(X, \widehat F_X\bigr) \pprec \bigl(Y, \widehat F_Y\bigr)$ est assurée par la proposition \ref{Prop:MaximaFaible} et son unicité par la remarque \ref{Rem:UniCommu}.
Par ailleurs, le lemme \ref{Le:T2} appliqué à $\bigl(Z, \widehat F_Z\bigr)$, $\bigl(Y, \widehat F_Y\bigr)$ et au point isolé $y$ assure l'existence de chemins de $C_Y$ associés à $\widehat F_Y$ et à $\widehat F_Z$.

 \item[\textbullet]
si $C_X$ est homéomorphe au disque, utilisons la proposition \ref{PropDef:123} : il existe une isotopie restreinte $(Z,J_Z)$ telle que $(Z, \widehat F_Z)$ soit associé à $(Z, J_Z)$.
Si $w \in C_X$ est suffisamment proche de $y$ isolé dans $Z$, alors $J_Z(w)$ est contenu dans l'anneau $C_Y$.
Parmi les relèvements $\widehat F_Y$ possibles (il en existe d'après la proposition \ref{Prop:MaximaFaible}), on choisit celui associé à $J_Z(w)$ et on obtient la propriété recherchée.
\end{itemize}

\textbf{Le relèvement $\bigl(Y, \widehat F_Y\bigr)$ construit ne dépend pas de $Z_1$.}
Cette propriété est immédiate lorsque $C_X$ n'est pas homéomorphe au disque, puisque l'on a alors unicité d'un relèvement $\bigl(Y, \widehat F_Y\bigr)$ vérifiant $\bigl(X, \widehat F_X\bigr) \pprec \bigl(Y, \widehat F_Y\bigr)$ (remarque \ref{Rem:UniCommu}).
Nous allons maintenant chercher à la démontrer lorsque $C_X$ est homéomorphe au disque.

Soient $Z_1 \subset X$ et $Z_1' \subset X$ deux sous-ensembles finis de cardinaux supérieur ou égal à $2$.
On note $Z = Z_1 \cup \{y\}$, $Z' = Z_1' \cup \{y\}$ puis $Z_1'' = Z_1 \cup Z_1'$ et $Z'' = Z_1'' \cup \{y\}$.
La construction précédente fournit trois relèvements $\bigl(Y, \widehat F_Y\bigr)$, $\bigl(Y, \widehat F'_Y\bigr)$ et $\bigl(Y, \widehat F''_Y\bigr)$ tels que les chemins de $C_Y$ associés à $\widehat F_Y$, $\widehat F'_Y$ et $\widehat F''_Y$ soient également associés respectivement à $\widehat F_Z$, $\widehat F_{Z'}$ et $\widehat F_{Z''}$.
D'après le lemme \ref{Le:T2} appliqué à $\bigl(Z, \widehat F_Z\bigr)$ et $\bigl(Z'', \widehat F_{Z''}\bigr)$ d'une part, à $\bigl(Z', \widehat F_{Z'}\bigr)$ et $\bigl(Z'', \widehat F_{Z''}\bigr)$ d'autre part, le point isolé étant $y$, on a $\bigl(Z, \widehat F_Z\bigr) \pprec \bigl(Z'', \widehat F_{Z''}\bigr)$ et $\bigl(Z', \widehat F_{Z'}\bigr) \pprec \bigl(Z'', \widehat F_{Z''}\bigr)$.
Un chemin de $C_Y$ associé à $\widehat F''_Y$ est associé à $\widehat F_{Z''}$ mais il est donc également associé à $\widehat F_Z$ et $\widehat F_{Z'}$ donc à $\widehat F_Y$ et $\widehat F'_Y$ d'après la construction de $\widehat F_Y$ et $\widehat F'_Y$.
On en déduit $\widehat F'_Y = \widehat F_Y$ comme espéré.
\end{proof}

Nous pouvons maintenant démontrer la proposition \refbis{Prop:MaximaFaible}.

\begin{proof}[Démonstration de la proposition \refbis{Prop:MaximaFaible} dans le cas général]
On suppose ici $X$ infini, le cas fini ayant déjà été traité.
La proposition \ref{Prop:T3} fournit un relèvement $\bigl(Y, \widehat F_Y\bigr)$ bien déterminé.
Nous considérons désormais ce relèvement.

Nous pouvons maintenant  \textbf{vérifier que $(Y,\widehat F_Y)$ appartient à $\mc R'$}.
 Considérons donc un sous-ensemble fermé $Z$ de $Y = X \cup \{y\}$ et montrons l'existence de $(Z,\widehat F_Z) \in \mc R$ vérifiant $(Z,\widehat F_Z) \pprec (Y,\widehat F_Y)$.
Il n'est pas restrictif de supposer que $Z$ est de cardinal supérieur ou égal à $3$.
En effet, si $Z$ est de cardinal au plus $2$, on peut choisir $Z'$ de cardinal $3$ vérifiant $Z \subset Z' \subset Y$.
Si l'on sait construire $(Z', \widehat F_{Z'}) \in \mc R$ vérifiant $(Z', \widehat F_{Z'}) \pprec (Y, \widehat F_Y)$, alors $Z'$ est fortement non enlacé (et $\widehat F_{Z'}$ est unique) et on peut donc trouver $(Z, \widehat F_Z) \in \mc R$ vérifiant $(Z, \widehat F_Z) \pprec (Z', \widehat F_{Z'}) \pprec (Y, \widehat F_Y)$.
On considère donc $Z$ de cardinal supérieur ou égal à $3$.
Deux cas se présentent.

Le \textbf{premier cas} est celui où $Z$ est un sous-ensemble fermé de $X$.
  Dans ce cas, puisque $(X,\widehat F_X)$ appartient à $\mc R'$, il existe $(Z,\widehat F_Z) \in \mc R$ vérifiant $(Z,\widehat F_Z) \pprec (X,\widehat F_X)$ et par transitivité $(Z,\widehat F_Z) \pprec (Y,\widehat F_Y)$.

Le \textbf{deuxième cas} est celui où $Z$ est de la forme $Z = Z_1 \cup \{y\}$ avec $Z_1$ sous-ensemble fermé de~$X$.
On note $C_{Z_1}$ la composante connexe de $N_{Z_1}$ qui contient $y$ et $C_Z = C_{Z_1} \setminus \{y\}$ .
Comme précédemment, puisque $\bigl(X, \widehat F_X\bigr)$ appartient à $\mc R'$, on sait qu'il existe $(Z_1,\widehat F_{Z_1}) \in \mc R$ vérifiant $(Z_1,\widehat F_{Z_1}) \pprec (X, \widehat F_X) \pprec (Y,\widehat F_Y)$.
La proposition \ref{Prop:MaximaFaible} assure l'existence d'un relèvement $\bigl(Z, \widehat F_Z\bigr) \in \mc R$ vérifiant $\bigl(Z_1, \widehat F_{Z_1}\bigr) \pprec \bigl(Z, \widehat F_Z\bigr)$.
On a vu qu'il y a en général unicité de ce relèvement sauf dans le cas particulier où $C_Z$ est homéomorphe à l'anneau et il faut alors choisir $\widehat F_Z$ convenablement.
Soit $A \subset C_Y$ un anneau autour de $y$, de sorte que $A \cup \{y\}$ soit un disque.
Lorsque $C_Z$ est homéomorphe à l'anneau, on choisit le revêtement $\widehat F_Z$ tel que les chemins contenus dans $A$ et associés à $\widehat F_Y$ (il en existe) soient associés à $\widehat F_Z$.

Il faut maintenant vérifier que l'on a $(Z, \widehat F_Z) \pprec (Y, \widehat F_Y)$.
D'après la remarque \ref{Rem:R03}, il suffit de montrer que dans chaque composante de $N_Y$, il existe un chemin associé à $\widehat F_Y$ et $\widehat F_Z$.
Plus précisément,
il suffit de montrer ce résultat dans chaque composante connexe de $C_Z \cap N_X$. En effet, par construction de $\widehat F_Z$, dans les composantes connexes de $N_Z$ autres que $C_Z$, les chemins associés à $\widehat F_Z$ sont les chemins associés à $\widehat F_{Z_1}$ qui vérifie $(Z_1, \widehat F_{Z_1}) \pprec (Y, \widehat F_Y)$.

\begin{enumerate}
 \item
Intéressons-nous d'abord à la \textbf{composante connexe $U^*=C_X$ de $C_Z \cap N_X$} bordant $y$.
Plusieurs cas se présentent :
\begin{enumerate}
\item
si $C_Z$ est homéomorphe à l'anneau, on a choisi $\widehat F_Z$ de sorte que $U^*$ contienne un chemin adapté à $\widehat F_Y$ et à $\widehat F_Z$ ;

\item
si $C_Z$ n'est pas homéomorphe à l'anneau, on considère un sous-ensemble fini $W_1 \subset Z_1$ de cardinal au moins 2 et on note $W = W_1 \cup \{y\}$.
Considérons un anneau $A \subset N_Y$ autour de $y$ (tel que $A \cup \{y\}$ soit un disque).
On sait qu'il existe un unique relèvement $(W, \widehat F_W) \in \mc R$ et que $A$ contient un chemin $\alpha$ associé à $\widehat F_W$.

Considérons tout d'abord l'inclusion $W \subset Z$ avec $W$ contenant $y$ isolé dans $Z$, $C_Z$ non homéomorphe à l'anneau, de même que la composante connexe de $N_W$ bordant $y$ (car $W_1$ est fini de cardinal au moins $2$). Le lemme \ref{Le:T2} assure que $\alpha$ est associé à $\widehat F_Z$.

Considérons ensuite l'inclusion $W \subset Y$.
D'après la proposition \ref{Prop:T3}, $\alpha$ est aussi associé à $\widehat F_Y$.
On a obtenu un chemin $\alpha \in U^*$ associé à $\widehat F_Z$ et $\widehat F_Y$.
\end{enumerate}

 \item
Intéressons-nous maintenant aux autres composantes connexes de $C_Z \cap N_X$ et choisissons donc une composante connexe $U$ de $C_Z \cap N_X$ autre que $U^*$.
D'après la proposition \ref{PropDef:123}, il existe une isotopie restreinte $(Y,I_Y)$ associée à $\widehat F_Y$.
Il faut montrer que $I_Y(z)$ est également associée à $\widehat F_Z$, pour un point $z \in U$.

Pour cela, considérons un point $w \in \partial U$ accessible par un arc simple $\beta$ à valeurs dans $U$ (voir figure \ref{Fig:RR}).
On considère de même un point $w' \in \partial U^*$ distinct de $w$ et accessible par un arc simple $\beta'$ à valeurs dans $U^*$ issu de $y$.
Notons $W_1 = Z_1 \cup \{ w, w'\}$ et $W = Z \cup \{w,w'\} = Z_1 \cup \{w,w',y\}$.
Puisque $(X, \widehat F_X)$ appartient à $\mc R'$, il existe $(W_1, \widehat F_{W_1}) \in \mc R$ vérifiant $(W_1, \widehat F_{W_1}) \pprec (X, \widehat F_X)$.
La proposition \ref{Prop:MaximaFaible} assure l'existence d'un relèvement $(W, \widehat F_W) \in \mc R$ vérifiant $\bigl(W_1, \widehat F_{W_1}\bigr) \pprec \bigl(W, \widehat F_{W}\bigr)$.
Il existe comme précédemment une isotopie restreinte $(W, I_{W})$ associée à $(W, \widehat F_{W})$ (qui se prolonge en $(W_1, I_{W_1})$ associée à $\widehat F_{W_1}$).

\Fig{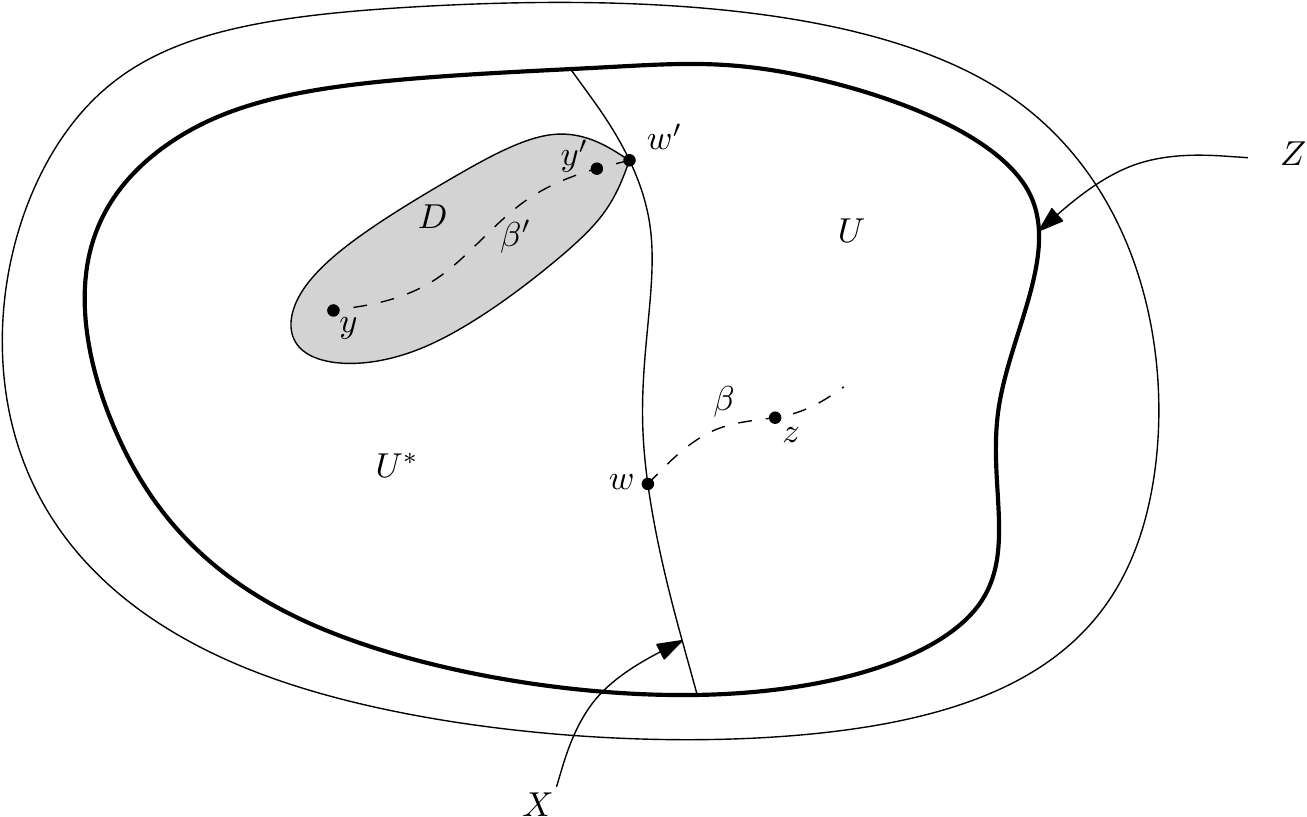}{Fig:RR}{Démonstration de la proposition \refbis{Prop:MaximaFaible}}

Si $z \in \beta$ est suffisamment proche de $w$, alors les deux arcs $I_Y(z)$ et $I_{W}(z)$ sont disjoints de $\beta'$.
On peut alors trouver un disque ouvert $D \subset N_X$ vérifiant $w' \in \partial D$ et $\beta' \subset \bigl(D \cup \{w'\}\bigr)$ et tel que $I_X(z)$ et $I_{W}(z)$ ne rencontrent pas $\overline D$.
Or on a $(W_1, \widehat F_{W_1}) \pprec (X, \widehat F_X)$ donc $I_Y(z)$ est homotope dans $N_{W_1}$ à $I_{W}(z)$.
Si $y' \in \beta'$ est suffisamment proche de $w'$, $I_Y(z)$ et $I_W(z)$ sont homotopes dans $N_{W_1} \setminus \{y'\}$.
D'après le lemme \ref{Le:Alex}, les chemins $I_Y(z)$ et $I_W(z)$ sont donc homotopes à extrémités fixées dans $N_{W_1} \setminus \{y\} = N_W$.
Or $I_W(z)$ est associé à $\widehat F_W$ donc $I_Y(z)$ l'est aussi.
Ainsi, on a bien trouvé un chemin $I_Y(z)$ associé à $\widehat F_Y$ et $\widehat F_Z$ comme attendu.
\qedhere
\end{enumerate}
\end{proof}

\section{Existence de relèvements maximaux}
\label{sec:induc}

\subsection{Suites croissantes de $(\mc R, \pprec)$}

Ce paragraphe est consacré à la démonstration des propositions \ref{Prop:existmax} et \refbis{Prop:existmax}.
L'existence de relèvements maximaux dans $(\mc R,\pprec)$ repose sur le lemme de Zorn.
L'essentiel des difficultés consiste à montrer la proposition suivante :

\begin{Prop} \label{Prop:Zorn3}
 Soit $F \in \Homeo_*(M)$ et
 soit $\bigl(X_n,\widehat F_n\bigr)_{n \in \N}$ une suite strictement croissante dans $(\mc R,\pprec)$.
 Alors la suite admet un majorant $\bigl(X_\infty,\widehat F_\infty\bigr) \in \mc R$ vérifiant $\displaystyle X_\infty = \overline{\bigcup_{n \in \N} X_n}$.
\end{Prop}

\PropBis{Zorn3}
\begin{PropZorn3}
 Soit $F \in \Homeo_*(M)$ et
 soit $\bigl(X_n,\widehat F_n\bigr)_{n \in \N}$ une suite strictement croissante dans $(\mc R',\pprec)$.
 Alors la suite admet un majorant $\bigl(X_\infty,\widehat F_\infty\bigr) \in \mc R'$ vérifiant $\displaystyle X_\infty = \overline{\bigcup_{n \in \N} X_n}$.
\end{PropZorn3}

\begin{proof}[Démonstration de la proposition \ref{Prop:existmax}]
 Rappelons que l'on suppose donné $(X,\widehat F_X) \in \mc R$.
 On utilise le lemme de Zorn avec l'ensemble $\mc E$ des relèvements $\bigl(Y, \widehat F_Y\bigr) \in \mc R$ qui vérifient $\bigl(X, \widehat F_X\bigr) \pprec \bigl(Y, \widehat F_Y\bigr)$.
 On va donc montrer que $\mc R$ est inductif.
 Soit $\mc F = \bigl(Y_j,\widehat F_j\bigr)_{j \in J}$ une famille de  $\mc E$ totalement ordonnée.
 Si $\mc F$ contient un élément maximal, la démonstration est achevée.
 Sinon, l'ensemble $Y = \bigcup\limits_{j \in J} Y_j$ est distinct de chaque $Y_j$ d'après la remarque \ref{Rem:R02}.
 Comme $M$ est séparable, on peut trouver une suite croissante d'ensembles finis $(X_n)_{n \in \N}$ vérifiant :
 \[
  \bigcup_{n \in \N} X_n \;\subset\; Y
  \qquad \text{et} \qquad
  \overline{ \bigcup_{n \in \N} X_n } = \overline Y.
 \]
 Or, $\mc F$ est totalement ordonnée.
 Pour tout entier $n \in \N$, comme $X_n$ est fini et contenu dans $Y$, il existe donc un ensemble $Y_{j_n}$ contenant $X_n$.
 Nous allons maintenant extraire de la suite $\bigl(Y_{j_n}, \widehat F_{j_n}\bigr)_{n \in \N}$ une sous-suite strictement croissante dans $(\mc R, \pprec)$.
 
 Pour tout entier $n \in \N$, puisque $Y_{j_n}$ est distinct de $Y$, il existe un entier $m > n$ tel que l'on n'ait pas $X_m \subset Y_{j_n}$, ni \textit{a fortiori} $Y_{j_m} \subset Y_{j_n}$.
 Comme $\mc F$ est totalement ordonnée, on en déduit $\bigl(Y_{j_n},\widehat F_{j_n}\bigr) \pprec \bigl(Y_{j_m},\widehat F_{j_m}\bigr)$ avec $Y_{j_n} \neq Y_{j_m}$.
 Il existe donc une sous-suite $\bigl(Y_{j_{\varphi(n)}},\widehat F_{j_{\varphi(n)}}\bigr)_{n \in \N}$ de la suite $\bigl(Y_{j_n},\widehat F_{j_n}\bigr)_{n \in \N}$, strictement croissante dans $(\mc R,\pprec)$.
 D'après la proposition \ref{Prop:Zorn3}, cette sous-suite est majorée par un élément de la forme $\bigl(\overline Y,\widehat F_\infty\bigr) \in \mc R$.
 
 Il reste à vérifier que $\bigl(\overline Y,\widehat F_\infty\bigr)$ majore bien tous les éléments de $\mc F$.
 Soit $\bigl(Y', \widehat F'\bigr) \in \mc F$ un tel élément.
 Si on avait $X_n \subset Y'$ pour tout $n \in \N$, on aurait aussi $\overline Y \subset Y'$, ce dernier ensemble étant fermé.
 Comme $\mc F$ est totalement ordonnée, $\bigl(Y',\widehat F'\bigr)$ serait alors un élément maximal de $\mc F$, ce qui contredit l'hypothèse.
 Il existe donc un entier $n \in \N$ tel que l'on n'ait pas $X_n \subset Y'$ ce qui entraîne $\bigl(Y',\widehat F'\bigr) \pprec \bigl(Y_{j_n},\widehat F_{j_n}\bigr) \pprec \bigl(Y_\infty,\widehat F_\infty\bigr)$.
 Le relèvement $\bigl(Y_\infty,\widehat F_\infty\bigr)$ est bien un majorant de $\mc F$.
\end{proof}

\begin{proof}[Démonstration de la proposition \refbis{Prop:existmax}]
 La démonstration est identique à celle de la proposition \ref{Prop:existmax} en remplaçant $\mc R$ par $\mc R'$ et en utilisant la proposition \refbis{Prop:Zorn3} au lieu de la proposition \ref{Prop:Zorn3}.
\end{proof}

Nous allons maintenant démontrer la proposition \ref{Prop:Zorn3}.
Une partie de la preuve (le lemme \ref{Le:NP4}) sera reportée au paragraphe suivant.

\begin{proof}[Démonstration de la proposition \ref{Prop:Zorn3}]
Introduisons l'ensemble
\[ X_\infty = \overline{\bigcup_{n \in \N} X_n} \qquad
 \text{et son complémentaire $N_\infty = M \setminus X_\infty$.}
\]
 Le revêtement universel de $N_\infty$ est obtenu par réunion des revêtements universels de chaque composante connexe de $N_\infty$.
De plus, M. Brown et J.M. Kister \cite{BK84} ont montré que $F$ fixe les composantes connexes de $N_\infty$.
\footnote{Dans le cas particulier qui nous intéresse, on peut redémontrer ce résultat facilement.
  En effet, supposons par l'absurde qu'il existe un point $x \in N_\infty$ tel que $x$ et $F(x)$ soient dans deux composantes distinctes $U$ et $V$ de $N_\infty$.
  On peut alors choisir trois arcs disjoints (sauf en $x$) dans $U \cup X_\infty$ de $x$ vers trois points de $X_\infty$ dont l'ordre cyclique détermine l'orientation près de $x$.
  Les images dans $V \cup X_\infty$ de ces arcs déterminent une orientation opposée près de $F(x)$ ce qui contredit que $F$ préserve l'orientation.
}
Il suffit donc de définir $\widehat F_\infty$ sur chaque composante connexe de $N_\infty$.

Plaçons-nous dans une composante connexe $N'_\infty$ de $N_\infty$.
On peut alors écrire $N'_\infty$ comme la réunion d'une suite croissante de surfaces à bord, compactes et connexes : $N'_\infty = \bigcup_{n\in\N} M'_n$, par exemple parce que $M$ est dénombrable à l'infini et triangulable (théorème de Rad\'o \cite{Rado25} ou \cite{Moise77}~page~60) et tout compact rencontre un nombre fini de triangles. 
De plus, $M \backslash M'_n$ a toujours un nombre fini de composantes connexes.
Ensuite, pour tout $n \in \N$, nous notons $M_n$ la réunion de $M'_n$ et des composantes connexes de $M \backslash M'_n$ qui ne rencontrent pas $X_\infty$.
Ainsi, d'une part $N'_\infty$ est la réunion croissante des sous-variétés connexes à bord $M_n$, d'autre part chaque composante connexe de $M \backslash M_n$ rencontre~$X_\infty$.
Quitte à extraire une sous-suite, on peut également supposer $F(M_n) \subset M_{n+1}$.

Les composantes connexes de $M \setminus M_n$ sont en nombre fini et chacune intersecte $X_\infty$.
Il existe donc une suite extraite $\left(X_{\varphi(n)}\right)_{n \in \N}$ telle que pour tout $n \in \N$, chaque composante connexe de $M \backslash M_n$ intersecte au moins un point de $X_{\varphi(n)}$.
Pour éviter d'alourdir inutilement les notations, nous pouvons remplacer la suite $(X_n,\widehat F_n)$ par la suite $\left(X_{\varphi(n)},\widehat F_{\varphi(n)}\right)_{n \in \N}$ ou en d'autres termes supposer que les composantes connexes de $M \backslash M_n$ contiennent chacune au moins un point de $X_n$.

Pour tout $n \in \N$, on introduit également $N_n = M \backslash X_n$ et $\widehat\pi_n : \widehat N_n \rightarrow \nobreak N_n$ son revêtement universel.
Enfin, on choisit arbitrairement un point de la pré-image de $M_0$ par $\widehat\pi_n$ et pour tout $m \in \N$, on note $\widehat M_m^n \subset \widehat N_n$ la composante connexe de la pré-image de $M_m$ par $\widehat\pi_n$ contenant ce point.

\begin{Rem} \label{Rem:Msc}
 Si $n \geq m$, la composante $\widehat M^n_m$ est une surface à bord dont les bords sont des droites (ou un cercle dans le cas particulier où $M_m$ est un disque).
 De fait, $\widehat M^n_m$ est simplement connexe.
 En effet, $\widehat N_n$ est simplement connexe et toutes les composantes connexes de $M \backslash M_m$ contiennent un point de $X_m$ et donc de $X_n$.
\end{Rem}

Nous avons vu précédemment que s'il existe un relèvement $(Y,\widehat F_Y) \in \mc R$, tout lacet de $N_Y$ est librement homotope à son image (proposition \ref{PropDef:123}).
Une première étape de la preuve de la proposition \ref{Prop:Zorn3} sera de prouver le résultat \textit{a priori} plus faible suivant que nous allons admettre momentanément.

\begin{Lemme} \label{Le:NP4}
Pour tout lacet $\Gamma \subset N'_\infty$, il existe une homotopie libre dans $N'_\infty$ entre $\Gamma$ et $F(\Gamma)$ qui, pour tout $n \in \N$, est associée à $\widehat F_n$.
\end{Lemme}

 Achevons la preuve de la proposition \ref{Prop:Zorn3} et construisons le relèvement~$\widehat F_\infty$. 
 Choisissons un lacet $\Gamma \subset N'_\infty$ basé en un point $z \in N'_\infty$.
 D'après le lemme \ref{Le:NP4}, il existe une homotopie libre entre $\Gamma$ et $F(\Gamma)$ qui, pour tout $n \in \N$, soit associée à $\widehat F_n$.
 Notons $\alpha$ le chemin décrit par $z$ le long de cette homotopie.
 Nous allons montrer que ce chemin est adapté à $X_\infty$.
 
 Pour cela, considérons un lacet $\gamma \subset N'_\infty$ basé en $z$.
 Il existe un entier $n \in \N$ tel que les lacets $\gamma$ et $F(\gamma)$ ainsi que le chemin $\alpha$ soient tous contenus dans $M_n$.
 Choisissons un relèvement $\widehat z_n \in \widehat M^n_n$ de $z$ à $\widehat N_n$.
 Alors les lacets $\gamma$ et $\alpha.F(\gamma).\alpha^{-}$ se relèvent en deux chemins issus de $\widehat z_n$ entièrement contenus dans $\widehat M_n$.
 Or $\alpha$, associé à $\widehat F_n$, est donc adapté à $X_n$ d'après la proposition \ref{Prop:PlusPrecis}.
 Il en résulte que les lacets $\gamma$ et $\alpha.F(\gamma).\alpha^{-}$ sont homotopes dans $N_n$.
 Leurs relèvements dans $\widehat N_n$ d'origine $\widehat z_n$ ont donc même extrémité.
 Comme $\widehat M_n$ est simplement connexe d'après la remarque \ref{Rem:Msc}, ils sont homotopes dans $\widehat M_n$ à extrémités fixées.
 Cette homotopie se projette sur une homotopie dans $M_n$ donc dans $N'_\infty$ entre les lacets $\gamma$ et $\alpha.F(\gamma).\alpha^{-}$.
 Nous avons bien montré que $\alpha$ est adapté à $X_\infty$.
 
 La proposition \ref{Prop:PlusPrecis} assure l'existence de $\widehat F_\infty$ et le chemin $\alpha$ sera associé à $\widehat F_\infty$.
 Il reste à vérifier, pour tout $n \in \N$, la relation $(X_n,\widehat F_n) \pprec (X_\infty,\widehat F_\infty)$.
 Pour cela, considérons un entier $n \in \N$ ; des inclusions
 \begin{gather*}
  \forall p \geq n \qquad X_p \subset \left( X_n \cup \widehat\pi_{X_n}( \Fix(\widehat F_n)) \right),\\
  \text{on tire} \qquad\qquad
  X_n \subset X_\infty \subset \left( X_n \cup \widehat\pi_{X_n}( \Fix(\widehat F_n)) \right).
 \end{gather*}
Par ailleurs, le chemin $\alpha$ est associé à $\widehat F_\infty$ et à $\widehat F_n$, ce qui montre $(X_n,\widehat F_n) \pprec (X_\infty,\widehat F_\infty)$ grâce à la remarque \ref{Rem:R03}.
\end{proof}

\begin{proof}[Démonstration de la proposition \refbis{Prop:Zorn3}]
Il faut maintenant montrer que si l'on suppose que pour tout entier $n$ le relèvement $(X_n,\widehat F_n)$ appartient à $\mc R'$, alors $(X_\infty,\widehat F_\infty)$ que l'on vient de construire appartient lui aussi à $\mc R'$.
Pour cela, considérons un sous-ensemble fermé $Y \subset X_\infty$.
Il suffit de montrer que le chemin $\alpha$ est adapté à $Y$.
En utilisant la proposition \ref{Prop:PlusPrecis}, on en déduira qu'il existe un relèvement $(Y, \widehat F_Y) \in \mc R$ et on pourra vérifier $(Y, \widehat F_Y) \pprec (X_\infty,\widehat F_\infty)$.

Soit un lacet $\Gamma \subset N_Y$ basé en $z$, origine de $\alpha$.
Notons $\Delta$ le lacet $\alpha.F(\Gamma).\alpha^-$.
Comme $\Delta$ est compact et $Y$ est fermé, on peut construire un voisinage compact de $\Delta$ qui ne rencontre pas $Y$ et qui est une variété à bord.
Notons $M'$ la variété à bord obtenue en réunissant ce voisinage avec les composantes connexes de son complémentaire qui ne rencontrent pas $Y$.
Remarquons que $M \setminus M'$ a un nombre fini de composantes connexes.
On choisit un point $x_i \in \bigcup_{j \in \N} X_j$ dans chacune de ces composantes connexes ; on obtient une famille finie $(x_i)_{1 \leq i \leq p}$.
Il existe donc $j_0 \in \N$ tel que $X_{j_0}$ contienne tous les $(x_i)_{1 \leq i \leq p}$.
Comme $\alpha$ est adapté à $X_{j_0}$, les lacets $\Gamma$ et $\Delta$ sont homotopes dans $N_{j_0} = M \setminus X_{j_0}$.

Notons $\widehat\pi_{j_0} : \widehat N_{j_0} \to N_{j_0}$ le revêtement universel de $N_{j_0}$.
Relevons $\Gamma$ à~$\widehat N_{j_0}$ en un chemin $\widehat\Gamma$ et relevons ensuite $\Delta^-$ en $\widehat\Delta^-$ à partir de l'extrémité de~$\widehat\Gamma$.
Puisque $\Gamma$ et $\Delta$ sont homotopes dans $N_{j_0}$, $\widehat\Gamma.\widehat\Delta^-$ est un lacet.
On note~$\widehat M'$ la composante connexe de $\widehat\pi_{j_0}^{-1}(M')$ qui contient $\widehat\Gamma$.
Par construction de~$M'$ et de la famille $(x_i)_{1 \leq i \leq p}$, l'ensemble $\widehat M'$ est simplement connexe.
Le lacet $\widehat\Gamma.\widehat\Delta^-$ est donc contractile dans $\widehat M'$.
De plus, $\widehat M'$ ne rencontre pas la préimage par $\widehat\pi_{j_0}$ de $Y \setminus X_{j_0}$.
Il en résulte que $\widehat\Gamma.\widehat\Delta^-$ est contractile dans la préimage de $N_Y$ donc $\Gamma.\Delta$ est contractile dans $N_Y$.
On a montré, comme annoncé, que $\alpha$ est adapté à $Y$ d'où l'existence d'un relèvement $(Y, \widehat F_Y) \in \mc R$ d'après la proposition \ref{Prop:PlusPrecis}.

Nous voulons montrer $(Y, \widehat F_Y) \pprec (X_\infty, \widehat F_\infty)$.
En utilisant la remarque \ref{Rem:R03} sachant que $\alpha$ est adapté à $Y$ et à $X_\infty$, il reste à montrer :
\[
 Y \subset X_\infty \subset \left( Y \cup \widehat\pi_Y(\Fix(\widehat F_Y)) \right)
\]
et plus précisément la seconde inclusion.
Par continuité de $\widehat F_Y$, cela revient à montrer :
\[
 \forall n \in \N \qquad X_n \subset \left( Y \cup \widehat\pi_Y(\Fix(\widehat F_Y)) \right)
\]

Considérons donc un entier $n \in \N$ et un point $x \in X_n \setminus (X_n \cap Y)$ et montrons $x \in \widehat\pi_Y(\Fix(\widehat F_Y))$.
Pour cela, considérons un chemin $\beta \subset N_Y$ de $z$ vers $x$ et le lacet $\gamma = \beta^-.\alpha.F(\beta)$ basé en $x$.
On va faire le même raisonnement que ci-dessus.
On commence par considérer un voisinage compact de $\gamma$ dans $N_Y$ et la variété à bord $M''$ obtenue en réunissant ce voisinage avec les composantes connexes de son complémentaire qui ne rencontrent pas $Y$.
Dans chaque composante connexe de $M \setminus M''$, on choisit un point $x'_i \in \bigcup_{j \in \N} X_j$ et on obtient une famille finie $(x'_i)_{1 \leq i \leq p}$.
On note $Y'$ l'ensemble des points $(x'_i)_{1 \leq i \leq p}$ et on choisit $m \geq n$ vérifiant $Y' \subset X_m$.

Puisque $(X_m, \widehat F_{X_m})$ appartient à $\mc R'$, il existe un relèvement $\widehat F_{Y'}$ vérifiant $\left( Y', \widehat F_{Y'}\right) \pprec \left(X_m, \widehat F_{X_m}\right)$.
Soit $\widehat x' \in \widehat N_{Y'}$ un relevé de $x$.
De $\left( Y', \widehat F_{Y'}\right) \pprec \left(X_m, \widehat F_{X_m}\right)$, on déduit $\widehat F_{Y'}(\widehat x') = \widehat x'$.
Relevons $\beta$ et $\alpha$ à $\widehat N_{Y'}$ en $\widehat\beta$ d'extrémité $\widehat x'$ et $\widehat\alpha$ de même origine que $\widehat\beta$.
Comme $\alpha$ est associé à $\widehat F_{Y'}$, le lacet $\gamma$ se relève à $\widehat N_{Y'}$ en $\widehat\gamma = \widehat\beta^-.\widehat\alpha.\widehat F_{Y'}(\widehat\beta)$ d'extrémité $\widehat F_{Y'}(\widehat x') = \widehat x'$.
C'est donc un lacet et $\gamma$ est homotopiquement trivial dans $N_{Y'}$.

Enfin, remarquons comme précédemment que, par construction de $Y'$, la composante connexe $\widehat M''$ de $\widehat\pi_{Y'}^{-1}(M'')$ qui contient $\widehat\gamma$ est simplement connexe et ne rencontre pas la préimage de de $Y \setminus (Y \cap Y')$ par $\widehat\pi_{Y'}$.
Ainsi $\widehat\gamma$ est contractile dans la préimage de $N_Y$ et $\gamma$ est contractile dans $N_Y$.
Pour tout relevé $\widehat x \subset \widehat N_Y$ de $x$, on peut relever $\gamma$ en un lacet de $\widehat N_Y$ d'origine~$\widehat x$.
Comme $\alpha$ est associé à $\widehat F_Y$, on en déduit $\widehat F_Y(\widehat x) = \widehat x$ ce qui achève la démonstration.
\end{proof}

\subsection{Homotopie entre tout lacet et son image}

Pour démontrer le lemme \ref{Le:NP4}, nous allons avoir besoin d'un lemme préliminaire.

\begin{Lemme} \label{Le:P5}
Soit $\Gamma\subset N'_{\infty}$ un lacet non contractile basé en un point $z \in M$.
On se donne un entier $m$ suffisamment grand pour que les propriétés  suivantes soient vérifiées~:

\begin{itemize}
 \item 
les lacets $\Gamma$, $F(\Gamma)$, $F^2(\Gamma)$ sont inclus dans $M_m$~;
 \item
il existe un chemin $\delta$ joignant $z$ à $F(z)$ tel que $\delta$ et $F(\delta)$ sont inclus dans $M_m$. 
\end{itemize}

Alors il existe une homotopie libre $\phi_m$ de $\Gamma$ à $F(\Gamma)$ à support dans $M_m$ et associée à $\widehat F_m$.
\end{Lemme}

\begin{proof}
Notons $\widehat G_m$ le groupe des automorphismes de revêtement du revêtement universel $\widehat\pi_m~: \widehat N_m\to N_m$.
Considérons un relèvement $\widehat \Gamma$ de $\Gamma$ à $\widehat N_m$ issu d'un relevé $\widehat z\in \widehat M_m^m$ de $z$ et commençons par montrer que $\widehat\Gamma$ n'est pas un lacet.

Par l'absurde, si c'était le cas, alors $\widehat\Gamma$ serait un lacet contenu dans $\widehat M_m^m$ qui est simplement connexe d'après la remarque \ref{Rem:Msc}.
Ainsi $\Gamma$ serait contractile dans $M_m$, ce qui contredit les hypothèses du lemme.
Ainsi $\widehat\Gamma$ n'est pas un lacet.
Il en résulte que le chemin $\widehat \Gamma$ joint $\widehat z$ à un point de la forme $\widehat T(\widehat z)$, où $\widehat T\in\widehat G_m$ est distinct de l'identité.

Montrons maintenant que $\widehat  F_m(\widehat \Gamma)$ est contenu dans $\widehat M_m^m$ et donc en particulier que $\widehat F_m(\widehat z)$ appartient à $\widehat M_m^m$.  
Par hypothèse, $\widehat \Gamma$ est inclus dans $\widehat M_m^m$.
On en déduit que $\widehat T(\widehat z)\in \widehat M_m^m$, et donc que   $\widehat T(\widehat M_m^m)= \widehat  M_m^m$.
Notons $\widehat \delta$ le relèvement de $\delta$ à $\widehat N_m$ d'extrémité $\widehat F_m(\widehat z)$.
L'origine de $\widehat \delta$ s'écrit $\widehat S(\widehat z)$, où $\widehat S\in \widehat G_m$ (voir la partie gauche de la figure \ref{Fig:LemmeDouble}).
La composante connexe de $\widehat \pi^{-1}(M_m)$ qui contient $\widehat F_m(\widehat z)$ est donc $\widehat S(\widehat M_m^m)$.
Nous allons supposer que  $\widehat S(\widehat M_m^m)\not=\widehat M_m^m$ et aboutir à une contradiction.

\Fig{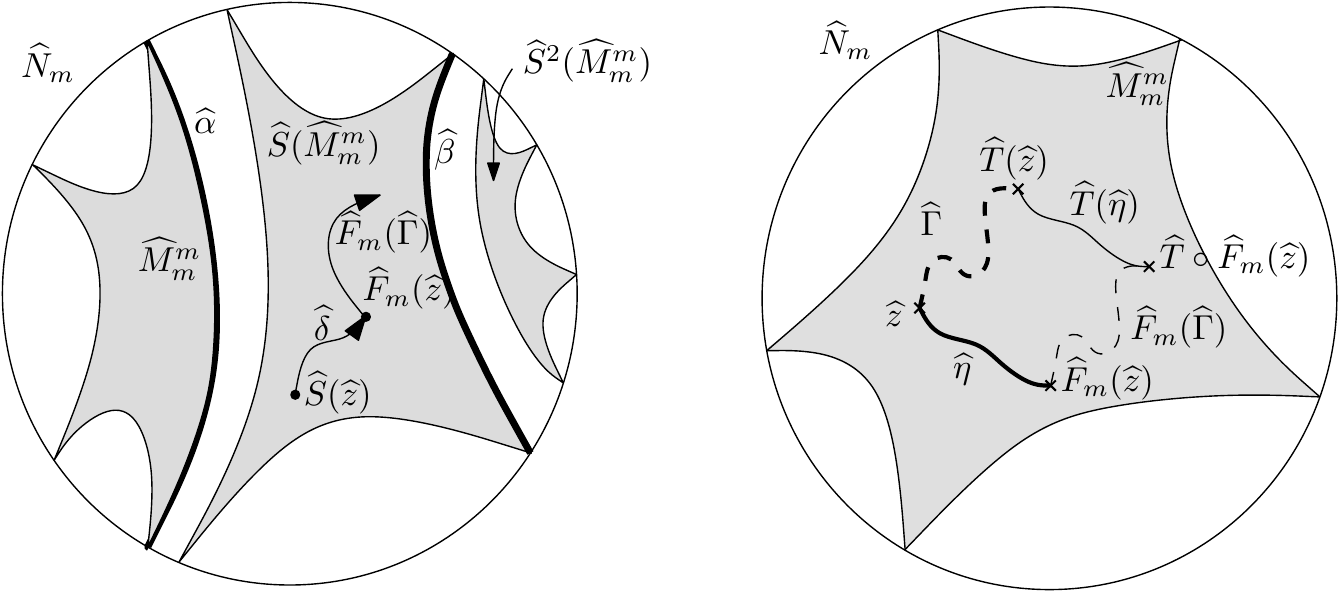}{Fig:LemmeDouble}{Démonstration du lemme \ref{Le:P5}}

Par hypothèse, nous savons que $\widehat F_m(\widehat \Gamma)$ est inclus dans $\widehat S(\widehat M_m^m)$.
Puisque $\widehat F_m(\widehat \Gamma)$ joint $\widehat F_m(\widehat z)$ à $\widehat F_m\circ\widehat T(\widehat z)=\widehat T \circ \widehat F_m(\widehat z)$, on en déduit que $\widehat T$ stabilise   $\widehat S(\widehat M_m^m)$.
Pour les mêmes raisons, cet automorphisme stabilise également $\widehat S^2(\widehat M_m^m)$.
En effet,  $\widehat F_m^2(\widehat \Gamma)$ joint $\widehat F^2_m(\widehat z)$  à $\widehat F^2_m\circ\widehat T(\widehat z)=\widehat T \circ \widehat F^2_m(\widehat z)$  et est inclus dans une composante connexe de $\widehat \pi^{-1}(M_m)$ ; il faut montrer que c'est  $\widehat S^2(\widehat M_m^m)$.
Il suffit pour cela de remarquer que  $\widehat F^2_m(\widehat z)$ appartient à $\widehat S^2(\widehat M_m^m)$.
En effet, par hypothèse, on sait que $\widehat F_m(\widehat  \delta)$ appartient à une composante connexe de $\widehat \pi^{-1}(M_m)$, que son origine $\widehat F_m\circ \widehat S(\widehat z)=\widehat S\circ \widehat F_m(\widehat z)$ est incluse dans $\widehat S^2(\widehat M_m^m)$ et que son extrémité est  $\widehat F^2_m(\widehat z)$.
  
Puisque $M_m$ est une surface à bord, toute composante connexe $\gamma$ de la préimage $\pi_m^{-1}(\partial M_m)$ de la frontière de $M_m$ est une droite topologique proprement plongée dans $\widehat N_m$ et son stabilisateur dans $\widehat G_m$ est un groupe monogène infini  $\widehat G(\gamma)$.
De plus pour tout $\widehat U\in \widehat G_m$, s'il existe un entier $r\neq0$ tel que $\widehat U^r\in \widehat G(\gamma)$, alors $\widehat U\in\widehat G(\gamma)$.
Ceci implique que si $\gamma'$ est une autre composante connexe de de $\pi_m^{-1}(\partial M_m)$, alors $\widehat G(\gamma)=  \widehat G(\gamma')$  ou  $\widehat G(\gamma)\cap \widehat G(\gamma') = \{\id_{\widehat N_m}\}$.
Il existe une composante connexe $\widehat \alpha$ de la frontière de $ \widehat  M_m^m$ qui sépare  l'intérieur de $\widehat M_m^m$ de $\widehat S( \widehat M_m^m)$.
On sait alors que $\widehat \beta= \widehat S(\widehat \alpha)$ est la composante connexe de la frontière de $\widehat S(\widehat  M_m^m)$ qui sépare  l'intérieur de $\widehat S(\widehat M_m^m)$ de $\widehat S^2( \widehat M_m^m)$.
L'automorphisme $\widehat T$, stabilisant  $\widehat M_m^m$ et $\widehat S( \widehat M_m^m)$, stabilise également  $\widehat \alpha$.
De même, puisqu'il stabilise $ \widehat S(\widehat M_m^m)$ et  $\widehat S^2( \widehat M_m^m)$, il stabilise également  $\widehat \beta$.
De $\widehat T \neq \id_{\widehat N_m}$, on déduit $\widehat G(\alpha)=  \widehat G(\beta)$.

Puisque $\widehat G(\beta)$ est l'image de $\widehat G(\alpha)$ par l'automorphisme intérieur $\widehat U\mapsto \widehat S\circ \widehat U\circ \widehat S^{-1}$, cet automorphisme laisse invariant $\widehat G(\alpha)$.
Ainsi, si $\widehat U$ est un générateur de $\widehat G(\alpha)$, alors $\widehat S\circ \widehat U\circ \widehat S^{-1}$ en est également un, et on a donc $\widehat S\circ \widehat U = \widehat U\circ \widehat S$ ou $\widehat S\circ \widehat U = \widehat U^{-1}\circ \widehat S$ :
\begin{itemize}
 \item 
dans le premier cas, Epstein a montré que $\widehat S$ et $\widehat U$ engendrent un groupe monogène infini (\cite{Epstein66}, Lemma 4.3), ce qui n'est possible que si $\widehat S$ est dans $\widehat G(\alpha)$ ;
 \item
dans le second cas, on obtient $\widehat U^{-1} = \widehat S \circ \widehat U \circ \widehat S^{-1}$ puis :
\[ \widehat U \circ \widehat S^2 = (\widehat U^{-1})^{-1} \circ \widehat S^2 = (\widehat S \circ \widehat U \circ \widehat S^{-1})^{-1} \circ \widehat S^2 = \widehat S \circ \widehat U^{-1} \circ \widehat S = \widehat S^2\circ \widehat U \]
et on peut conclure au même résultat.
\end{itemize}
La contradiction provient de ce que $\widehat S$ ne stabilise pas $\widehat\alpha$.
On vient de montrer $\widehat S(\widehat M_m^m) = \widehat M_m^m$ et donc, comme annoncé, $\widehat F_m(\widehat z)$ est contenu dans~$\widehat M_m^m$.

Poursuivons la démonstration : nous pouvons donc choisir un chemin $\widehat\eta$ de $\widehat z$ à $\widehat F_m(\widehat z)$ contenu dans $\widehat M_m^m$ qui se projette en un chemin $\eta$ de $z$ à $F(z)$ (voir la partie droite de la figure \ref{Fig:LemmeDouble}).
Le chemin $\widehat F_m(\widehat\Gamma)$ a pour origine $\widehat F_m(\widehat z)$ et pour extrémité $\widehat F_m \circ \widehat T(\widehat z)$ qui n'est autre que $\widehat T \circ \widehat F_m(\widehat z)$ extrémité de $\widehat T(\widehat\eta)$.
On peut donc construire une homotopie libre $\widehat \phi_m$ entre $\widehat\Gamma$ et $\widehat F_m(\widehat\Gamma)$ telle que $\widehat z$ parcoure $\widehat\eta$ le long de l'homotopie.
Cette homotopie $\widehat \phi_m$ se projette en une homotopie libre $\phi_m$ entre $\Gamma$ et $F_m(\Gamma)$.
Comme $\widehat M_m^m$ est simplement connexe d'après la remarque \ref{Rem:Msc}, on peut choisir $\phi_m$ à support dans $M_m$.
De plus, comme $\eta$ est associé à $\widehat F_m$, l'homotopie $\widehat \phi_m$ est associée à $\widehat F_m$ et à support dans $M_m$.
\end{proof}

\begin{proof}[Démonstration du lemme \ref{Le:NP4}]
  Considérons un lacet $\Gamma \subset N'_\infty$, non contractile dans $N'_\infty$, basé en un point $z \in N'_\infty$.
 On considère également un chemin $\delta$ de $z$ à $F(z)$ contenu dans $N'_\infty$.
 Il existe un entier $m \in \N$ tel que $\Gamma$, $F(\Gamma)$, $F^2(\Gamma)$, $\delta$ et $F(\delta)$ soient contenus dans $M_m$.
 On en déduit que $\Gamma$ n'est pas contractile dans $N_m$.
 D'après le lemme \ref{Le:P5},  pour tout $n \geq m$, il existe une homotopie libre $\phi_n$ de $\Gamma$ à $F(\Gamma)$ à support dans $N_\infty$ et associée à $\widehat F_n$.
 Considérons l'homotopie $\phi_m$ ; nous allons prouver qu'elle est associée à $\widehat F_n$ pour tout $n \in \N$.
 
 Notons $\alpha_m = \phi_m(z)$.
 Le chemin $\alpha_m$ est associé à $\widehat F_m$ donc également associé à $\widehat F_n$ pour tout $n \leq m$ car la suite $\bigl(X_p, \widehat F_p\bigr)_{p \in \N}$ est croissante pour~$\pprec$.
 Il en résulte que pour tout $n \leq m$, $\phi_m$ est associée à $\widehat F_n$.
 
 Pour tout $n \geq m$, on choisit un relèvement $\widehat\Gamma_n \subset \widehat M_m^n$ de $\Gamma$ à $\widehat N_n$.
 Notons $\alpha_n = \phi_n(z)$.
Ce chemin $\alpha_n$ est associé à $\widehat F_n$ donc à $\widehat F_m$, toujours d'après la relation d'ordre $\pprec$.
D'après la remarque \ref{Rem:UniChAdapt}, $\alpha_n$ est homotope à $\alpha_m$ à extrémités fixées dans $N_m$.
Il reste à montrer qu'ils sont en réalité homotopes dans $N'_\infty$.

Pour cela, considérons le lacet $\alpha_n.\alpha_m^-$.
Notons $\widehat z$ un relèvement de $z$ au revêtement universel $\widehat N'_\infty$ de $N'_\infty$.
Relevons le chemin $\alpha_n$ en un chemin $\widehat\alpha_n$ issu de $\widehat z$ puis $\alpha_m^-$ en un chemin $\widehat\alpha_m^-$ issu de l'extrémité de $\widehat\alpha_n$.
Il existe un automorphisme $\widehat T$ du revêtement $\widehat\pi : \widehat N'_\infty \to N'_\infty$ tel que l'extrémité de $\widehat\alpha_n . \widehat\alpha_m^-$ soit $\widehat T(\widehat z)$.
De même relevons $\Gamma$ en un chemin $\widehat\Gamma \subset \widehat N'_\infty$ d'origine $\widehat z$.
Notons $\widehat U$ l'automorphisme du revêtement $\widehat\pi : \widehat N'_\infty \to N'_\infty$ envoyant l'origine de $\widehat\Gamma$ sur son extrémité.

On retrouve les arguments employés à la fin de la démonstration de la proposition \refbis{Prop:MaximaFaible}.
L'homotopie $\phi_n.\phi_m^-$, entre $\Gamma$ et lui-même, se relève en une homotopie entre le chemin $\widehat\Gamma$ et le chemin $\widehat T\bigl(\widehat\Gamma\bigr)$ d'extrémité $\widehat T  \circ \widehat U (\widehat z)$.
D'autre part, l'extrémité $\widehat U(\widehat z)$ de $\widehat\Gamma$ décrit $\widehat U(\widehat\alpha_n.\widehat\alpha_m^-)$ le long de l'isotopie et est donc envoyée sur $\widehat U \circ \widehat T(\widehat z)$.
On en déduit $\widehat U \circ \widehat T = \widehat T \circ \widehat U$.

Si $N'_\infty$ n'est ni le tore, ni l'anneau, il en découle que $U$ et $T$ appartiennent à un même sous-groupe monogène du groupe des automorphismes de revêtement.
Ce résultat perdure si $N'_\infty$ est l'anneau car le groupe des automorphismes de revêtement est alors monogène.
Enfin, $N'_\infty$ ne peut être le tore car cela signifierait que $X_\infty$ est vide ce qui est impossible compte tenu des hypothèses de la proposition \ref{Prop:Zorn3}.
Finalement, il existe toujours $(k,l) \in \Z^2$ tel que $\widehat T^k = \widehat U^l$.

Cela signifie que $(\alpha_n.\alpha_m^-)^k$ est librement homotope à $\Gamma^l$ dans $N'_\infty$ donc dans $N_m$.
Or le premier lacet est homotopiquement trivial dans $N_m$ tandis que le second ne l'est que si $l=0$ auquel cas $\widehat T$ est l'identité.
On obtient alors que $\alpha_n$ et $\alpha_m$ sont homotopes dans $N'_\infty$ à extrémités fixées.
Comme~$\alpha_n$ est associé à $\widehat F_n$, $\alpha_m$ est également associé à $\widehat F_n$.
D'après la remarque~\ref{Rem:Associe}, $\phi_m$ est donc associée à $\widehat F_n$.
\end{proof}

\section{Lacets contractiles} \label{sec:contrac}

L'objet de cette partie est la démonstration de la proposition \ref{PropTh:XyNonEnlace}.
Nous en donnerons une preuve topologique qui consiste à ramener le lacet contractile sur un point.
Le cas le plus simple est celui où le lacet est entièrement contenu dans un disque de $N_X$ ; nous l'aborderons en premier avant de nous intéresser au cas général.

\begin{Lemme}\label{Le:PPhi}
 Soit $D$ un disque topologique fermé de $M$.

\noindent Alors il existe une famille $(\phi^D_{x,x',s})_{(x,x',s) \in \mathring D \times \mathring D \times [0,1]}$ d'homéomorphismes de $M$ vérifiant :
\begin{enumerate}
 \item
l'application $(x,x',s) \mapsto \phi^D_{x,x',s}$ est continue de $\mathring D \times \mathring D \times [0,1]$ dans $\Homeo_*(M)$.
 \item
pour tout $(x,x',s) \in \mathring D \times \mathring D \times [0,1]$, la restriction $(\phi^D_{x,x',s})_{|M\setminus D}$ est l'identité de $M \setminus D$ ;
 \item
pour tout $(x,s) \in \mathring D \times [0,1]$, l'homéomorphisme $\phi^D_{x,x,s}$ est l'identité de $M$ ;
 \item
pour tout $(x,x') \in \mathring D \times \mathring D$, l'homéomorphisme $\phi^D_{x,x',0}$ est l'identité de $M$ ;
 \item
pour tout $(x,x') \in \mathring D \times \mathring D$, on a $\phi^D_{x,x',1}(x) = x'$.
\end{enumerate}
\end{Lemme}

\begin{proof}
Considérons un homéomorphisme $h : D \rightarrow \mathbb D$ entre $D$ et le disque unité fermé $\mathbb D$ de $\R^2$.
Considérons ensuite l'homéomorphisme $k$ du disque unité ouvert $\mathring{\mathbb D}$ sur $\R^2$ défini par $k(z) = \tan\left(\frac\pi2 |z|\right) z$ pour tout $z \in \mathring{\mathbb D}$.
Enfin, pour tout $u \in \R^2$, notons $t_u : x \mapsto x + u$ la translation affine de $\R^2$ de vecteur $u$.
Pour tout $(x,x',s) \in \mathring D \times \mathring D \times [0,1]$, on définit $\phi^D_{x,x',s}$ par :
\[
 \begin{cases}
 (\phi^D_{x,x',s})_{| \mathring D} = h^{-1} \circ k^{-1} \circ 
 t_{s.(k\circ h(x')-k\circ h(x))} \circ k \circ h\\
 (\phi^D_{x,x',s})_{| M \setminus \mathring D} = \id_{M \setminus \mathring D}.
 \end{cases}
\]
L'application $(\phi^D_{x,x',s})_{| \mathring D}$ est un homéomorphisme de $\mathring D$ qui se prolonge par continuité, ainsi que sa réciproque, en l'identité sur le bord $\partial D$ de $D$.
On en déduit que l'on a bien défini un homéomorphisme $\phi^D_{x,x',s}$ de $M$.
Grâce à la continuité de $(x,x',s) \mapsto (\phi^D_{x,x',s})_{| \mathring D}$, on obtient l'assertion 1.
Les trois assertions suivantes sont faciles à vérifier.
Enfin, un calcul rapide conduit à $(\phi^D_{x,x',1})_{| \mathring D}(x) = x'$ qui n'est autre que la dernière assertion.
\end{proof}

\begin{Lemme} \label{Le:DisqueSuite}
 Soit $D_N$ un disque topologique fermé d'une surface $N$,
 $x$ et $x'$ deux points de $D_N$,
 $\gamma$ et $\gamma'$ deux chemins de l'intérieur de $D_N$ d'origines $x$ et d'extrémités $x'$.
 Il existe une isotopie $J = (G_t)_{t \in [0,1]}$ de l'identité de $D_N$ à elle-même vérifiant :
 \begin{enumerate}
  \item $J$ est à support dans $D_N$ ;
  \item pour tout $t \in [0,1]$, on a
    $G_t\circ\gamma(t) = \gamma'(t)$ ;
  \item $J$ est homotope à extrémités fixées à l'isotopie identité relativement au complémentaire de $D_N$.
 \end{enumerate}
\end{Lemme}

\begin{proof}
 On utilise le lemme précédent et pour tout $t \in [0,1]$, on définit $G_t$ par :
\[
 G_t = \phi^{D_N}_{\gamma(t),\gamma'(t),1}
\]
L'application $H$ définie par :
\[
 H(s,t) = \phi^{D_N}_{\gamma(t),\gamma'(t),s}
\]
réalise bien une homotopie entre l'isotopie identité $H(0,t)$ et $G_t = H(1,t)$, relativement au complémentaire de $D_N$.
\end{proof}

\begin{Rem}
 Dans le cas particulier où le lacet $\gamma = I(y)$ est contenu dans un disque fermé $D_N \subset N_X$, la proposition \ref{PropTh:XyNonEnlace} est un corollaire immédiat du lemme \ref{Le:DisqueSuite}.
 En effet, celui-ci assure l'existence d'une isotopie $J = (G_t)_{t \in [0,1]}$ de l'identité de $N_X$ à elle-même à support dans $D_N$ et vérifiant, pour tout $t \in [0,1]$, $G_t \circ \gamma(t) = y$.
 L'isotopie restreinte $(X,I')$ définie par $I' = J \circ I$ convient.
 
 Le principe de la démonstration qui suit est de se ramener à ce cas particulier.
\end{Rem}

\begin{proof}[Démonstration de la proposition \ref{PropTh:XyNonEnlace}]
Notons $(X,I) \in \mc I$ sous la forme $I = (F_t)_{t \in [0,1]}$.
Considérons le carré $K = [0,1]^2$ et un paramétrage $\theta_0 : [0,1] \to \partial K$ du bord de $K$ vérifiant $\theta_0(0) = \theta_0(1) = (0,0)$.
Soit $\varphi : \partial K \to N_X$
définie par $\varphi \circ \theta_0(t) = F_t(y)$ pour tout $t \in [0,1]$.
Comme $I(y)$ est un lacet, $\varphi$ est correctement définie et continue sur $\partial K$.
De plus le lacet $I(y)$ est contractile, donc $\varphi$ se prolonge en une application $\varphi : K \rightarrow N_X$ continue (nous conservons la notation $\varphi$ après prolongement).

Pour tout $n \in \N$ avec $n \geq 2$, on peut découper le carré $K$ en $n^2$ «~petits carrés~» de côté $\frac1n$ de la forme $\left[ \frac an, \frac{a+1}n\right] \times \left[ \frac bn, \frac{b+1}n\right]$ avec $a \in \ent0{n-1}$ et $b \in \ent0{n-1}$.
Comme $\varphi$ est continue sur le compact $K$, il existe un entier $n \geq 2$ tel que l'image par $\varphi$ de chacun des $n^2$ «~petits carrés~» soit contenue dans l'intérieur d'un disque topologique fermé de $M$.
On choisit désormais une telle valeur de $n$ et on note $L$ la réunion des $n^2$ disques de $M$ associés.

Nous allons maintenant construire une famille décroissante  $(K_i)_{0 \leq i \leq n^2}$ de sous-ensembles de $K$ telle que pour tout $i \in \ent0{n^2-1}$, $K_i$ soit la réunion, qui est simplement connexe, de $n^2-i$ «~petits carrés~» fermés, définie de la façon suivante :
\begin{itemize}
 \item on pose $K_0 = K$ ;
 \item soit $i \in \ent1{n^2-1}$ ; supposons $K_{i-1}$ construit comme réunion de $n^2-i+1$ «~petits carrés~».
 Parmi ces «~petits carrés~», on considère, parmi ceux contenant les points d'ordonnées maximales, celui qui contient les points d'abscisses maximales et on le note $C_i$.
 On note $K_i$ la réunion des $n^2-i$ «~petits carrés~» restants.
 \item en répétant ce procédé, on «~enlève les petits carrés~» un à un en commençant par les plus hauts et pour chaque ligne les plus à droite.
 On obtient nécessairement $K_{n^2-1} = \left[0,\frac1n\right] \times \left[0,\frac1n\right]$, seul «~petit carré~» contenant $\{(0,0)\}$, que l'on note $C_{n^2}$.
 \item on pose enfin $K_{n^2} = \{(0,0)\}$.
\end{itemize}
On construit de même une suite de paramétrages $(\theta_i)_{i \in [0,1]}$ avec $\theta_i : [0,1] \to \partial K_i$ du bord de $K_i$ de la façon suivante :
\begin{itemize}
 \item $\theta_0$ a déjà été construit ;
 \item soit $i \in \ent1{n^2-1}$ ; supposons $\theta_{i-1}$ construit.
 Il existe un segment $[a_i,b_i] \subset ]0,1[$ tel que la restriction $(\theta_{i-1})_{|[a_i,b_i]}$ soit à valeurs dans $ \partial C_i$ et $(\theta_{i-1})_{|[0,a_i[ \cup ]b_i,1]}$ soit à valeurs dans $\partial K_{i-1} \setminus \partial C_i$.
 On choisit $\theta_i : [0,1] \to \partial K_i$ de sorte que l'on ait $(\theta_{i-1})_{|[0,a_i[ \cup ]b_i,1]} = (\theta_i)_{|[0,a_i[ \cup ]b_i,1]}$ ;
 \item on pose enfin $\theta_{n^2}(t) = (0,0)$ pour tout $t \in [0,1]$.
\end{itemize}

Soit $i \in \ent1{n^2}$.
L'ensemble $\varphi(C_i)$ est contenu dans l'intérieur d'un disque topologique fermé $D_i$.
Considérons les chemins $t \mapsto \theta_{i-1}(t)$ et $t \mapsto \theta_i(t)$ avec $t \in [a_i,b_i]$.
Ces deux chemins ont même extrémité et même origine ; ils paramètrent chacun une partie de $\partial C_i$.
Utilisons le lemme \ref{Le:PPhi} et définissons l'isotopie $J^i = (G^i_t)_{t \in [0,1]}$ de l'identité de $N_X$ à elle-même par :
\[
 \forall t \in [0,1] \qquad
 G^i_t =
   \phi^{D_i}_{\varphi\circ\theta_{i-1}(t),\varphi\circ\theta_i(t),1}
\]
On obtient alors $G^i_t \circ \varphi \circ \theta_{i-1}(t) = \varphi \circ \theta_i(t)$ pour tout $t \in [0,1]$.
Par récurrence immédiate, on trouve :
\begin{align*}
 \forall t \in [0,1], \quad
 &G^{n^2}_t \circ G^{n^2-1}_t \circ \ldots \circ
 G^2_t \circ G^1_t \circ \varphi \circ \theta_0(t)
 = \varphi \circ \theta_{n^2}(t)\\
 \text{ou encore}\quad\forall t \in [0,1], \quad
 &G^{n^2}_t \circ G^{n^2-1}_t \circ \ldots \circ
 G^2_t \circ G^1_t \circ F_t(y) = y.
\end{align*}
Définissons l'isotopie restreinte $(X,I')$ en posant $I' = J^{n^2} \circ \ldots \circ J^1 \circ I$.
L'équation précédente assure que $I'$ fixe $y$ et
se met bien sous la forme $I' = J \circ I$ avec $J$ à support dans $L$.
\end{proof}

\bibliographystyle{persobst}
\bibliography{Enlacement}

\end{document}